\documentclass[aop, preprint, 12pt]{imsart}

\RequirePackage[OT1]{fontenc}
\RequirePackage{amsthm,amsmath}
\RequirePackage[numbers]{natbib}
\RequirePackage[colorlinks,citecolor=blue,urlcolor=blue]{hyperref}
\usepackage{amsfonts}
\usepackage{color}
\usepackage{amssymb}
\usepackage[english]{babel}
\usepackage[T1]{fontenc}
\usepackage{graphicx}
\usepackage{subfigure}
\usepackage{latexsym}
\usepackage{amstext}
\usepackage{mathrsfs}
\usepackage{hyperref}
\usepackage{dsfont}
\usepackage{graphics}
\usepackage{enumerate}
\usepackage{paralist}
\usepackage{color}
\usepackage{fullpage}
\newtheorem{prop}{Proposition}[section]
\newtheorem{rem}[prop]{Remark}

\numberwithin{equation}{section}

\newcommand{\beq}{\begin{eqnarray}}
\newcommand{\beqq}{\begin{eqnarray*}}
\newcommand{\eeq}{\end{eqnarray}}
\newcommand{\eeqq}{\end{eqnarray*}}

\usepackage{mathtools}
\usepackage[english]{babel}
\usepackage[utf8]{inputenc}
\usepackage{amsmath}
\usepackage{graphicx}
\usepackage[colorinlistoftodos]{todonotes}
\usepackage{polynom}
\usepackage{amsfonts}
\usepackage{dsfont}
\usepackage{amsthm}
\usepackage{hyperref}
\usepackage{graphicx}
\newtheorem{theorem}{Theorem}[section]

\usepackage{fullpage}
\usepackage[T1]{fontenc}
\usepackage{hyperref}
\usepackage{xcolor}
\definecolor{link-color}{rgb}{0.15,0.4,0.15}
\hypersetup{
  colorlinks,
  linkcolor = {link-color}, citecolor = {link-color},
}

\usepackage{graphicx}
\usepackage{hyperref}
\usepackage{amsmath,amsthm,amssymb}
\usepackage{bbm}
\usepackage{tabularx}

\newcommand{\NN}{\mathbb{N}}

\newenvironment{eqnarr}{\begin{IEEEeqnarray}{rCl}}{\end{IEEEeqnarray}\ignorespacesafterend}

\renewcommand{\eqref}[1]{\hyperref[#1]{(\ref*{#1})}}

    \def\beq{\begin{eqnarr}}
    \def\eeq{\end{eqnarr}}
    \def\beqq{\begin{eqnarray*}} 
    \def\eeqq{\end{eqnarray*}} 

        \def\d{{\rm d}}

    \def\d{{\textnormal d}}

\newcommand*{\pref}[1]{\hyperref[#1]{(\ref*{#1})}}
\newcommand*{\refpref}[2]{\hyperref[#2]{\ref*{#1}(\ref*{#2})}}

\startlocaldefs
\numberwithin{equation}{section}
\theoremstyle{plain}

\endlocaldefs

\begin{document}

\begin{frontmatter}
\title{Skeletal stochastic differential equations for continuous-state branching process}

\runtitle{Coupled SDEs for skeletally decomposed CSBPs}

\begin{aug}
%

\author{\fnms Dorottya Fekete\thanksref{t1}\ead[label=e1]{d.fekete@bath.ac.uk, a.kyprianou@bath.ac.uk}},
\author{\fnms Joaquin Fontbona\thanksref{t2}\ead[label=e2]{fontbona@dim.uchile.cl}}
\and
\author{\fnms{Andreas E. Kyprianou}\thanksref{t3}\ead[label=e1]{d.fekete@bath.ac.uk, a.kyprianou@bath.ac.uk}}

\ead[label=e3]{third@somewhere.com}

\thankstext{t2}{Supported by Basal-Conicyt Centre for Mathematical Modelling and Millenium Nucleus NC120062}
\thankstext{t1}{Supported  by a scholarship from the EPSRC Centre for Doctoral Training, SAMBa}
\thankstext{t3}{Supported by EPSRC grant EP/L002442/1}


\affiliation{University of Bath, Universidad de Chile and University of Bath}

\address{Department of Mathematical Sciences \\
University of Bath\\
 Claverton Down\\ Bath, BA2 7AY\\
 UK.
\printead{e1}
}

\address{Centre for Mathematical Modelling, \\
DIM CMM, \\
UMI 2807 UChile-CNRS,\\
 Universidad de Chile, \\
 Santiago, \\
 Chile.
\printead{e2}
}
\end{aug}

\begin{abstract}\hspace{0.1cm}
It is well understood that a supercritical continuous-state branching process (CSBP) is equal in law to a discrete continuous-time Galton Watson process (the {\it skeleton} of {\it prolific individuals}) whose edges are dressed in a Poissonian way with immigration  which initiates subcritical CSBPs ({\it non-prolific mass}). 

Equally well understood in the setting of CSBPs and superprocesses is the notion of a {\it spine or immortal particle} dressed in a Poissonian way with immigration which initiates copies of the original CSBP, which emerges when conditioning the process to survive eternally. 

In this article, we revisit these notions for CSBPs and put them in a common framework using the well-established language of (coupled) SDEs (cf. \cite{DL1, DL, BLeG}). In this way, we are able to deal simultaneously with all types of CSBPs (supercritical, critical and subcritical) as well as understanding how the skeletal representation becomes, in the sense of weak convergence, a spinal decomposition when conditioning on survival.

We have two principal motivations. The first is to prepare the way to expand the SDE approach to the spatial setting of superprocesses, where recent results have increasingly sought the use of skeletal decompositions to transfer results from the branching particle setting to the setting of measure valued processes; cf. \cite{KM-SP, eckhoff, piotr}. The second is to provide a pathwise decomposition of CSBPs in the spirit of  genealogical coding of CSBPs via L\'evy excursions in Duquesne and LeGall \cite{DLG} albeit precisely where the aforesaid  coding fails to work because the underlying CSBP is supercritical.

\end{abstract}

\begin{keyword}[class=MSC]
\kwd[Primary ]{60J80, 60H30}
\kwd{}
\kwd[; secondary ]{60G99}
\end{keyword}

\begin{keyword}
\kwd{Continuous-state branching processes, stochastic differential equations, skeletal decomposition, spine decomposition}
\end{keyword}

\end{frontmatter}

\section{Introduction}

In this article we are interested in  $X=(X_t, t\geq 0)$  a  continuous-state, finite-mean branching process (CSBP). In particular, this means that $X$ is a   $[0,\infty)$-valued strong Markov process with absorbing state at zero and with law on $\mathbb{D}([0,\infty),\mathbb{R})$ (the space of c\`adl\`ag mappings from $[0,\infty)$ to $\mathbb{R}$) given by $\mathbb{P}_x$ for each initial  state $x\geq 0$, such that $\mathbb{P}_{x+y}=\mathbb{P}_x*\mathbb{P}_y$. Here, $\mathbb{P}_{x+y}=\mathbb{P}_x*\mathbb{P}_y$ means that the sum of two independent processes, one issued from $x$ and the other issued from $y$, has the same law as the process issued from $x+y$. Its semigroup is characterised by the Laplace functional 
\begin{equation}
\mathbb{E}_x({\rm e}^{-\theta X_t}) = {\rm e}^{-  xu_t(\theta)}, \qquad x, \theta, t\geq 0,\label{semigroup}
\end{equation}
where 
$u_t(\theta)$ uniquely solves the evolution equation
\begin{equation}
u_t(\theta) + \int_0^t \psi(u_s(\theta) ){\rm d}s = \theta ,\qquad t\geq 0 .
\label{int-equa}
\end{equation}
Here, we assume that the so-called  branching mechanism
$\psi$ takes the form
\begin{equation}
\psi(\theta) = -\alpha \theta + \beta\theta^2 + \int_{(0,\infty )} ({\rm e}^{-\theta x} - 1 + \theta x)\Pi({\rm d}x),\,\, \theta \geq 0,
\label{mechanism}
\end{equation}
where $\alpha\in\mathbb{R}$, $\beta\geq 0$ and $\Pi$ is a measure concentrated on $(0,\infty)$ which satisfies $\int_{(0,\infty)}(x\wedge x^2)\Pi({\rm d}x)<\infty$. 
These restrictions on $\psi$ are very mild and only exclude the possibility of having a non-conservative process or processes which have an infinite mean growth rate. 

 We also assume for convenience that $-\psi$ is not the Laplace exponent of a subordinator (i.e. a Bernstein function), thereby ruling out the case that $X$ has monotone paths. It is easily checked that $\psi$ is an infinitely smooth convex function on $(0,\infty)$ with at most two roots in $[0,\infty)$. More precisely, $0$ is always a root, however if $\psi'(0+)<0$, then there is a second root in $(0,\infty)$. 

The process $X$ is henceforth referred to as a $\psi$-CSBP.  It is easily verified that   
\begin{equation}
\mathbb{E}_x[X_t] = x{\rm e}^{-\psi'(0+)t}, \qquad t,x\geq0. 
\label{meangrowth}
\end{equation}
The mean growth of the process is therefore characterised by $\psi'(0+)$ and accordingly we classify CSBPs by the value of this constant. We say that the $\psi$-CSBP is supercritical, critical or subcritical accordingly as $-\psi'(0+) = \alpha$ is strictly positive, equal to zero or strictly negative, respectively.

\medskip

It is known that the process $(X,\mathbb{P}_x)$, $x>0$, can also be represented as the unique strong solution to the 
stochastic differential equation (SDE) 
 \begin{align}\label{eqflow}
    X_t= x& + \alpha \int_{0}^{t} X_{s-}{\rm d}s 
                      + \sqrt{2\beta}  \int_{0}^{t}\int_{0}^{X_{s-}} W({\rm d}s,{\rm d}u) 
                      + \int_{0}^{t}\int_{0}^{\infty}\int_{0}^{X_{s-}}r \tilde{N}({\rm d}s, {\rm d}r, {\rm d}\nu), 
    \end{align}
 for $   x>0, t\geq 0,  $
where  $W({\rm d}s, {\rm d}u)$  is  a white noise process 
on $(0, \infty)^{2}$ based on the Lebesgue measure ${\rm d}s\otimes {\rm d}u $ and 
 $N({\rm d}s,{\rm d}r,{\rm d}\nu)$ is  a Poisson point process 
 on $[0, \infty)^{3}$ with intensity 
${\rm d}s\otimes \Pi({\rm d}r) \otimes {\rm d}\nu$. Moreover, we  denote by $\tilde{N}({\rm d}s,{\rm d}r,{\rm d}\nu)$ the compensated measure of $N({\rm d}s,{\rm d}r,{\rm d}\nu)$. See  \cite{DL1, DL, BLG} for this fact and further  properties of the above SDEs.

Through the representation of a CSBP as either a strong Markov process whose semi-group is characterised by an integral equation, or as a solution to an SDE,  
there are three fundamental probabilistic decompositions that play a crucial role in motivating the main results in this paper. These concern CSBPs conditioned to die out, CSBPs conditioned to survive and a path decomposition of the supercritical CSBPs.

\medskip

\noindent{\bf CSBPs conditioned  to die out.}  To understand what this means, let us momentarily recall that for all supercritical continuous-state branching processes (without immigration) the event $\{\lim_{t\to\infty} X_t =0\}$ occurs with positive probability. Moreover, for all $x\geq 0$,
\[
\mathbb{P}_x(\lim_{t\uparrow\infty} X_t =0) = {\rm e}^{-\lambda^* x},
\]
where $\lambda^*$ is the unique root on $(0,\infty)$ of the equation $\psi(\theta) = 0$. Note that $\psi$ is strictly convex with the property that $\psi(0) = 0$ and $\psi(+\infty) = \infty$, thereby ensuring that the root $\lambda^*>0$ exists; see Chapter 8 and 9 of \cite{Kbook} for further details. It is straightforward to show that
the law of $(X, \mathbb{P}_x)$ conditional on the event $\{\lim_{t\uparrow\infty} X_t =0\}$, say $\mathbb{P}^*_x$, agrees with the law of a $\psi^*$-CSBP, where
\begin{equation}\label{psi*}
\psi^*(\theta ) = \psi(\theta + \lambda^*).
\end{equation}
See for example \cite{Sheu}.

\medskip

\noindent{\bf CSBPs conditioned to survive.} The event $\{\lim_{t\to\infty}X_t=0\}$ can be categorised further according to whether its intersection with $\{X_t >0 \text{ for all }t\geq 0\}$ is empty or not. The classical work of Grey \cite{Grey} distinguishes between these two cases according to an integral test. Indeed, the intersection is empty if and only if 
\begin{equation}
\int^\infty\frac{1}{\psi(\theta)}\d u <\infty.
\label{grey}
\end{equation}
If we additionally assume that $-\psi'(0+)=\alpha \leq 0$, that is to say, the process is critical or subcritical, then it is known that the notion of conditioning the process to stay positive can be made rigorous through a limiting procedure. More precisely, if we write 
\[
\zeta = \inf\{t>0: X_t =0\},
\]
 then for all $A\in \mathcal{F}_t^X: = \sigma(X_s : s\leq t)$ and $x>0$,
\[
\mathbb{P}^\uparrow_x(A) :=\lim_{s\to\infty}\mathbb{P}_x(A| \zeta> t+s)
\]
is well defined as a probability measure and satisfies the Doob $h$-transform
\begin{equation}
\left.\frac{\d \mathbb{P}_x^\uparrow}{\d \mathbb{P}_x}\right|_{\mathcal{F}^X_t} ={\rm e}^{-\alpha t} \frac{X_t}{x}\mathbf{1}_{\{t<\zeta\}}.
\label{staypos}
\end{equation}

In addition, $(X, \mathbb{P}^\uparrow_x)$, $x>0$, has been shown to be equivalent in law to a process which has  a pathwise description which we give below. Before doing so, we need to introduce some more notation. To this end, define $N^*$ to be a Poisson random measure  
 on $[0,\infty)^2\times\mathbb{D}([0,\infty), \mathbb{R})$  with intensity measure $\d s\otimes r\Pi(\d r)\otimes\mathbb{P}_r(\d\omega)$. Moreover,  $\mathbb{Q}$ is the intensity, or `excursion' measure  on the space $\mathbb{D}([0,\infty), \mathbb{R})$  which satisfies
\[
\mathbb{Q}(1- {\rm e}^{-\theta \omega_t}) =- \frac{1}{x}\log\mathbb{E}_x({\rm e}^{-\theta X_t})= u_t(\theta),
\]
for $\theta,t\geq 0$. Here, the measure $\mathbb{Q}$ is the excursion measure  on the space $\mathbb{D}([0,\infty), \mathbb{R})$ associated to $\mathbb{P}_x$, $x>0$. See Theorems 3.10, 8.6 and 8.22 of \cite{Libook} and \cite{ElKR, LeG, DK, DL14,ZLimm} for further details. 
We can accordingly build a Poisson point process $N^{\rm c}$ on $[0,\infty)\times \mathbb{D}([0,\infty), \mathbb{R})$ with intensity ${\color{black} 2\beta\d s\otimes  \mathbb{Q}(\d \omega)}$. Then, for $x>0$, $(X, \mathbb{P}^\uparrow_x)$ is equal in law to the stochastic process
\begin{equation}
\Lambda_t = X'_t + \int_0^t \int_{\mathbb{D}([0,\infty), \mathbb{R})} \omega_{t-s} N^{\rm c}(\d s, \d \omega)+ \int_0^t \int_0^\infty \int_{\mathbb{D}([0,\infty), \mathbb{R})} \omega_{t-s} N^*(\d s, \d r, \d \omega), \qquad t\geq 0,
\label{N*}
\end{equation}
where $X'$ has the law $\mathbb{P}_x$ and is independent of $N^{\rm c}$ and $N^*$, which are also independent of one another.
Intuitively, one can think of the process $(\Lambda_t, t\geq 0)$ as being the result of first running a subordinator 
\[
S_t = 2\beta t + \int_0^t \int_0^\infty r N^*(\d s,\d r), \qquad t\geq 0,
\]
where we have slightly abused our notation and written $N^*(\d s,\d r)$, $s, r>0$ in  place of $\int_{ \mathbb{D}([0,\infty),\mathbb{R} )} N^*(\d s,\d r, \d \omega )$, $s,r>0$.
The subordinator $(S_t, t\geq0)$ is usually referred to as the {\it spine}.

To explain the formula \eqref{N*}, in a Poissonian way, we dress the spine with versions of $X$ sampled under the excursion measure $\mathbb{Q}$. Moreover,  at each jump of $S$ we initiate an independent copy of $X$ with initial mass equal to the size of the jump of $S$. 
See for example (3.9) in \cite{Li98}, (4.3) in \cite{Li96}, (4.18) in \cite{Li2001} or the discussion in Section 12.3.2 of \cite{Kbook} or \cite{Libook}.
The reader is also referred to e.g. \cite{RR} or \cite{Lam1, Lam2} for further details of the notion of  a spine.

It turns out that one may also identify the effect of the change of measure within the context of the SDE setting. In \cite{FF}, it was shown that $(X, \mathbb{P}^\uparrow_x)$, $x>0$, offers the unique  strong solution to the SDE 
\begin{align}\label{surv}
 X_t = x&+\alpha\int_0^t X_{s-}\d s+\sqrt{2\beta}\int_0^t\int_0^{X_{s-}} W(\d s,\d u)
+\int_0^t\int_0^\infty\int_0^{X_{s-}} r \tilde{N}(\d s,\d r,\d u) \notag\\
&+\int_0^t\int_0^\infty r N^*(\d s,\d r) + 2\beta t,\qquad t\geq0,
 \end{align}
 where $W$, $N$ and $\tilde{N}$ are as in (\ref{eqflow}) and $N^*$ is as above, and all noises are independent.
 See also \cite{DL} and \cite{FL}.

\medskip

\noindent{\bf Skeletal path decomposition of supercritical CSBPs.} In \cite{DW, BFM} and \cite{BKM} it was shown  that the law of the process $X$, where $X$ is defined by \eqref{eqflow}, can be recovered from a supercritical continuous-time Galton--Watson process (GW), issued with a Poisson number of initial ancestors, and dressed in a Poissonian way using the law of the original process conditioned to become extinguished.

To be more precise, they showed that for each $x\geq 0$, $(X, \mathbb{P}_x)$  has the same law as the  process $(\Lambda_t, t \geq 0)$ which has the following pathwise construction.
First sample from a continuous-time Galton--Watson process with  branching rate $q = \psi'(\lambda^*)$ and offspring distribution $\{p_k: k\geq 0\}$ such that its  branching generator is given by
\begin{equation}
 q\left(\sum_{k\geq 0} p_k r^k - r\right) =  \frac{1}{\lambda^*}\psi(\lambda^*(1-r)), \qquad r\in[0,1].
\label{F}
\end{equation}
This continuous-time Galton--Watson process goes by the name of the {\it skeleton}  and offers the genealogy of {\it prolific individuals}, that is, individuals who have infinite genealogical lines of descent (cf. \cite{BFM}).
With the particular branching generator given by \eqref{F}, $p_0 = p_1 =0$, and for $k\geq 2$,     $p_k : = p_k ([0,\infty))$, where for $r\geq 0$,
\[
p_k({\rm d}r) =  \frac{1}{\lambda^* \psi'(\lambda^*)}\left\{\beta (\lambda^*)^2\delta_0({\rm d}r)\mathbf{1}_{\{k=2\}} + (\lambda^*)^k \frac{r^k}{k!} {\rm e}^{-\lambda^*r} \Pi({\rm d}r)\right\}.
\]
If we denote the aforesaid GW process by $Z = (Z_t, t\geq 0)$ then we shall also insist that $Z_0$ has a Poisson distribution with parameter $\lambda^*x$.
Next,  thinking of the trajectory of $Z$ as a graph, {\it dress} the life-lengths of $Z$ in such a way that a $\psi^*$-CSBP is independently grafted on to each edge of $Z$ at time $t$ with rate
\begin{equation}
2\beta {\rm d}\mathbb{Q}^* + \int_0^\infty y {\rm e}^{-\lambda^* y}\Pi({\rm d}y){\rm d}\mathbb{P}^*_y.
\label{dressingrate}
\end{equation}
Moreover, on the event that an individual dies and branches into $k\geq 2$ offspring, with probability $p_k({\rm d}x)$, an additional independent $\psi^*$-CSBP is grafted on to the branching point with initial mass $x\geq 0$. The quantity $\Lambda_t$ is now understood to be the total dressed mass present at time $t$ together with the mass present at time $t$ in an independent $\psi^*$-CSBP issued at time zero with initial mass $x$. Whilst it is clear that the pair $(Z,\Lambda)$ is Markovian, it is less clear that $\Lambda$ alone is Markovian. This must, however,   be the case given the conclusion that   $\Lambda$ and $X$ are equal in law. A key element in this respect is the non-trivial observation that, for each $t\geq 0$, the law of $Z_t$ given $\Lambda_t$ is that of a  Poisson  random variable
 with parameter $\lambda^*\Lambda_t$.

Such skeletal path decompositions for continuous-state branching processes, and spatial versions thereof, are by no means new. Examples include \cite{EO, SV1, SV2, EW, DW, BKM, HHK, KR, KPR}. 

\bigskip

In this paper our objective is to understand the relationship between the skeletal decompositions of the type described above  and the emergence of a spine on conditioning the process to survive. In particular, our tool of choice will be the use of SDE theory. 
The importance of this study is that it underlines a methodology that should carry over to the spatial  setting of superprocesses,  where recent results have increasingly sought the use of skeletal decompositions to transfer results from the branching particle setting to the setting of measure valued processes; cf. \cite{KM-SP, eckhoff, piotr, P}. In future work we hope to develop the SDE approach to skeletal decompositions in the spatial setting.  We also expect  this approach to be helpful in studying analogous decompositions in the setting of continuous state branching processes with competition  \cite{BFF,P}.
Moreover, although our method takes inspiration from the genealogical coding of CSBPs by L\'evy excursions, cf. Duquesne and LeGall \cite{DLG}, our approach appears to be applicable where the aforesaid method fails, namely supercritical processes.

\section{Main results}

In this section we summarise the main results of the paper. We have three main results. First, we provide a slightly more general family of skeletal decompositions in the spirit of \cite{DW}, albeit with milder assumptions and  that we use the language of SDEs. Second, taking lessons from this first result, we give a time-inhomogeneous skeletal decomposition, again using the language of SDEs, {\color{black} both for supercritical and (sub)critical CSBPs}. Nonetheless, our proof will take inspiration from classical ideas on the genealogical coding of CSBPs through the exploration of associated excursions of reflected L\'evy processes; see for example \cite{DLG} and the references therein. Finally, our third main result, shows that a straightforward limiting procedure in the SDE skeletal decomposition for  (sub)critical processes, which corresponds to conditioning on survival, reveals a weak solution to the SDE given in \eqref{surv}.  It will transpire that conditioning the process to survive until later and later times is equivalent to ``thinning'' the skeleton such that, in the limit,  we get the spine decomposition. The limiting procedure also intuitively explains how the spine emerges in the conditioned process as a consequence of stretching out the skeleton in the SDE decomposition of the (sub)critical processes.
\medskip

Before moving to the first main result,  let us introduce some more notation. The reader will note that it is very similar but, nonetheless, subtly different to previously introduced terms.
Define the Esscher transformed branching mechanism $\psi_\lambda:\mathbb{R}_+\rightarrow \mathbb{R}_+$ for $\theta\geq -\lambda$ and $\lambda\geq \lambda^*$ by
\begin{equation}
\psi_\lambda(\theta)=\psi(\theta+\lambda)-\psi(\lambda)=\psi'(\lambda)\theta+\beta\theta^2+\int_{(0,\infty)}\left({\rm e}^{-\theta x}-1+\theta x\right){\rm e}^{-\lambda x}\Pi(\d x),
\label{psicomp}
\end{equation}
where
\begin{equation*}
\psi'(\lambda)=-\alpha+2\lambda\beta+\int_{(0,\infty)}\left( 1-{\rm e}^{-\lambda x}\right)x\Pi(\d x)>0.
\end{equation*}
This is the branching mechanism of a subcritical branching process on account of the fact that $-\psi'_\lambda(0+) = -\psi'(\lambda)<0$.
{\color{black} Heuristically speaking, given that $\lambda\mapsto \psi'(\lambda)$ is increasing, the $\psi_\lambda$-CSBP becomes more and more subcritical as $\lambda$ increases.}

Next, we need the continuous time Galton Watson process parameterised by $\lambda\geq \lambda^*$, which has been seen before in e.g. \cite{DW} and agrees with the process described by \eqref{F} when $\lambda = \lambda^*$.  It branches at rate $\psi'(\lambda)$ and has  branching generator
given by 
\[
F_\lambda(s): =
\lambda^{-1} \psi((1-s)\lambda)
,\qquad  s\in[0,1], \lambda\geq \lambda^*.
\] That is to say,   writing $F_\lambda(s)$ as in the left-hand side of \eqref{F}, we now have $p_0={\psi(\lambda)}/{\lambda \psi'(\lambda)}$, $p_1=0$ and for $k\geq 2$,
\begin{equation*}
p_k=\frac{1}{\lambda\psi'(\lambda)}\left\lbrace \beta \lambda^2\mathbf{1}_{\{k=2\}}+\int_{(0,\infty)}\frac{(\lambda r)^k}{k!}{\rm e}^{-\lambda r}\Pi(\d r)\right\rbrace.
\end{equation*}
We will also use the family $(\eta_k(\cdot))_{k\geq 0}$ of  branch point immigration laws (conditional on the number of offspring at the branch point), where $\eta_1(\d r)=0$, $x\geq0$, and, otherwise, 
\begin{equation}
\label{etalambda}
\eta_k(\d r) =\frac{1}{p_k\lambda \psi'(\lambda)}\left\{ \psi(\lambda)\mathbf{1}_{\{ k=0\}}\delta_0(\d r)+\beta \lambda^2 \mathbf{1}_{\{ k=2\}}\delta_0(\d r)  +\mathbf{1}_{\{ k\geq 2\}} \frac{(\lambda r)^k}{k!} {\rm e}^{-\lambda r}\Pi(\d r)  \right\}, 
\end{equation}
for $r\geq 0$. Note in particular that, when $\lambda > \lambda^*$, there is the possibility that no offspring are allowed.
{\color{black} Since in this case some lines of descent are finite, the Galton-Watson process no longer represents the prolific individuals.}

Finally, we need to introduce a  series of driving sources of randomness for the SDE which will appear in Theorem \ref{coupled} below.  Let {\rm$   {\texttt N}^{0}$} be a Poisson random measure  on $[0, \infty)^{3}$ with intensity measure $\d s\otimes {\rm e}^{-\lambda r} \Pi(\d r) \otimes {\rm d}\nu$, {\rm  $\tilde{\texttt N}^0$} be the associated compensated version of {\rm ${\texttt N}^0$}, {\rm $ {\texttt N}^1(\d s,\d r, \d{j}) $}  be a Poisson point process on  $[0, \infty)^{2}\times \NN
 $ with intensity 
$  \d s\otimes r  {\rm e}^{-\lambda r} \Pi(\d r) \otimes\sharp(\d{j}),$ and finally let {\rm  ${\texttt N}^{2}(\d s,\d r, \d{k},\d{j})$} be a Poisson point process on  $[0, \infty)^{2}\times \NN_0\times\mathbb{N}
 $ with intensity 
$  \psi'(\lambda)\d s\otimes \eta_k(\d r)  \otimes p_k\sharp(\d k)\otimes \sharp(\d{j})
 $,  where $\mathbb{N}_0 = \{0\}\cup\mathbb{N}$ and  $\sharp(\d\ell)= \sum_{i\in \mathbb{N}_0}\delta_{i}(\d \ell)$, $\ell\geq 0$, denotes the counting measure on $\mathbb{N}_0$.  
As before $W({\rm d}s, {\rm d}u)$  will denote a white noise process 
on $(0, \infty)^{2}$ based on the Lebesgue measure ${\rm d}s\otimes {\rm d}u$.

\begin{theorem}\label{coupled} Suppose that $\psi$ corresponds to a supercritical branching mechanism (i.e. $\alpha>0$) and $\lambda\geq \lambda^*$. Consider the coupled system of SDEs 
 {\rm \begin{align}
 \left(
\begin{array}{l}
\Lambda_t\\
 Z_t \\
\end{array}
\right)   =    & \, \left(
\begin{array}{l}
\Lambda_0 \\
 Z_0 \\
\end{array}
\right) -\psi'(\lambda) \int_{0}^{t}
 \left( \begin{array}{l}
\Lambda_{s-} \\
 0 \\
\end{array}
\right) \d s
                      + \sqrt{2\beta}  \int_{0}^{t}\int_{0}^{\Lambda_{s-}} \left( \begin{array}{l}
1 \\
 0 \\
\end{array}
\right)
 W(\d s,\d u) \notag\\
 &  + \int_{0}^{t}\int_{0}^{\infty}\int_{0}^{ \Lambda_{s-}} \left( \begin{array}{l}
r  \\
 0 \\
\end{array}
\right)
  \tilde{{\texttt N}}^{0}(\d s, \d r, {\rm d}\nu)  \notag \\
                      &\,   + \int_{0}^{t}\int_{0}^{\infty}\int_{1}^{Z_{s -}}  \left( \begin{array}{l}
r \\
 0 \\
\end{array}
\right){\texttt N}^1(\d s,\d r,\d{j})
\notag\\                     &\,  
      + \int_{0}^{t}\int_{0}^{\infty}\int_{0}^{\infty}\int_{1}^{Z_{s- }} \left( \begin{array}{l}
\quad r \\
 k-1 \\
\end{array}
\right) {\texttt N}^{2}(\d s,\d r,\d{k},\d{j})
\notag\\                     &\,  
   + 2\beta  \int_{0}^{t}   \left( \begin{array}{l}
Z_{s-} \\
 0 \\
\end{array}
\right) \d s, \qquad t\geq 0.              
                          \label{coupledSDE}                          
    \end{align}}
The equation  \eqref{coupledSDE} has a unique strong solution for arbitrary ($\mathcal{F}_0$-measurable) initial values $\Lambda_0\geq 0$ and $Z_0\in \mathbb{N}_0$ (where $\mathcal{F}_t:=\sigma((\Lambda_s,Z_s):s\leq t)$).
Furthermore, under the assumption that $Z_0$ is an independent random variable which is Poisson distributed with intensity $\lambda \Lambda_0$ this unique solution satisfies the following:
\begin{description}
\item[(i)] For $t\geq0$,  conditional on $\mathcal{F}^\Lambda_t: = \sigma(\Lambda_s: s\leq t)$, $Z_t $ is Poisson distributed with intensity $\lambda \Lambda_t$;
\item[(ii)] 
The process $(\Lambda_t, t\geq 0)$ is Markovian and   a weak solution to \eqref{eqflow};
\item[(iii)] If $Z_0=0$, then $(\Lambda_t,t\geq 0)$ is a subcritical CSBP with branching mechanism $\psi_{\lambda}$.

\end{description}
\end{theorem}

If one focuses on the second element, $Z$, in the SDE \eqref{coupledSDE}, it can be seen that there is no dependency on the first element $\Lambda$. The converse is not true however. Indeed, the stochastic evolution for $Z$ is simply that of the continuous-time GW process with branching mechanism given by $F_\lambda(s)$, $s\in[0,1]$. Given the evolution of $Z$, the process $\Lambda$ here describes nothing more than the aggregation of a Poisson and branch-point dressing on $Z$ together with  an independent copy of a $\psi_\lambda$-CSBP. As is clear from \eqref{etalambda} this results in the skeleton $Z$ having the possibility of `dead ends' (no offspring). Of course if $\lambda = \lambda^*$ then this occurs with zero probability and the joint system of SDEs in (\ref{coupledSDE}) describes precisely the prolific skeleton decomposition.
In the spirit of \cite{DW}, albeit using different technology and in a continuum setting,  Theorem \ref{coupled} puts into a common   framework a parametric family of skeletal decompositions for supercritical processes. Related work also appears in \cite{AD, L}.

\begin{rem}\rm
Although we have assumed in the introduction that $\int_{(0,\infty)}(x\wedge x^2)\Pi({\rm d}x)<\infty$, the reader can verify from the proof that this is in fact not needed. Indeed, suppose that we relax the assumption on $\Pi$ to just $\int_{(0,\infty)}(1\wedge x^2)\Pi(\d x)<\infty$ and we take the branching mechanism in the form 
\[
\psi(\theta) = -\alpha \theta + \beta\theta^2 + \int_{(0,\infty )} ({\rm e}^{-\theta x} - 1 + \theta x\mathbf{1}_{\{x<1\}})\Pi({\rm d}x),\,\, \theta \geq 0,
\]
where $\psi'(0)<0$ and
\[
\int_{0+}\frac{1}{|\psi(\xi)|}{\rm d}\xi =\infty
\]
 to ensure conservative supercriticality. Then 
the  necessary adjustment one needs to make occurs, for example, in \eqref{eqflow}, where jumps of size greater than equal to 1 in the Poisson random measures ${N}$   is separated out without compensation. However, the form of \eqref{coupledSDE} remains the same as all jumps of $N^0$ can be compensated.
\end{rem}

 Our objective, however, is to go further and demonstrate how the SDE approach can also apply in the finite horizon setting. We do this below,
 {\color{black} but} we should remark that the skeletal decomposition is heavily motivated by the description of the CSBP genealogy using the so-called height process  in Duquesne and Le Gall \cite{DLG}. Indeed, {\color{black} for (sub)critical CSBPs} one may consider the conclusion of Theorem \ref{T}, below, as a rewording thereof.
 {\color{black} However, as the proof does not rely on the CSBP being (sub)critical, the same result holds in the supercritical case. Thus Theorem \ref{T} is also a time-inhomogeneous version of Theorem \ref{coupled} for supercritical CSBPs, which setting was not discussed in \cite{DLG}. }
 
{\color{black} Assume that $\psi$ is a branching mechanism that satisfies Grey's condition \eqref{grey}.}
  We fix a time marker $T>0$ and we want to describe a coupled system of SDEs in the spirit of \eqref{coupledSDE} in which the second component describes prolific genealogies to the time horizon $T$. 
  {\color{black} In other words, our aim is to provide an SDE decomposition of the CSBP along those individuals in the population who have a descendent at time $T$.}

  To this end, recall that $(u_t(\theta), t\geq 0)$ is given by \eqref{semigroup} and accordingly, for $t\geq 0$, $u_t(\infty) = -x^{-1}\log \mathbb{P}_x(X_t=0) $ gives the rate at which extinction has occurred by time $t$.
We need a Poisson random measure  $\texttt{N}^{0}_T$ on $[0,T)\times [0, \infty)^{2}$ with intensity  ${\rm d}s\otimes {\rm e}^{-u_{T-s}(\infty) r} \Pi({\rm d}r) \otimes {\rm d}\nu$,
 a Poisson process $ {\texttt N}^1_T$  on  $[0,T)\times [0, \infty)\times \NN_0
 $ with intensity 
 ${\rm d} s\otimes r  {\rm e}^{-u_{T-s}(\infty) r} \Pi({\rm d}r) \otimes\sharp({\rm d}j),$  and
 a Poisson process $ {\texttt N}^{2}_T({\rm d}s,{\rm d}r,{\rm d}k,{\rm d}j)$   on  $[0,T)\times[0, \infty)\times \NN_0\times\mathbb{N}$ with intensity 
 \begin{align*}
&  \left\{\frac{u_{T-s}(\infty)\psi'(u_{T-s}(\infty))-\psi(u_{T-s}(\infty))}{u_{T-s}(\infty)}  \right\}{\rm d}s\otimes   \eta_k^{T-s}({\rm d}r)\otimes p^{T-s}_k\sharp({\rm d}k)\otimes \sharp({\rm d}j),
 \end{align*}
 where, for $k\geq 2$, 
 \begin{equation}
 \eta_k^{T-s}({\rm d}r)= \frac{ \beta u_{T-s}^2(\infty)\mathbf{1}_{\{k=2\}}\delta_0({\rm d}r)+\left( u_{T-s}(\infty)r\right)^k {\rm e}^{-u_{T-s}(\infty) r}\Pi({\rm d}r)/k!}{p^{T-s}_k \left(u_{T-s}(\infty)\psi'(u_{T-s}(\infty))-\psi(u_{T-s}(\infty))\right)}, \qquad r\geq 0,
 \label{timeeta}
 \end{equation}
and $p^{T-s}_k$ is such that $p^{T-s}_0 =p^{T-s}_1 = 0$ and the remaining probabilities are computable by insisting that $\eta_k^{T-s}(\cdot)$ is itself a probability distribution for each $k\geq2$.

\begin{theorem}\label{T} Suppose that $\psi$ corresponds to a branching mechanism which satisfies Grey's condition \eqref{grey}. Fix a time horizon $T>0$ and 
consider the  coupled system of SDEs 
{\rm  \begin{align}
 \left(
\begin{array}{l}
\Lambda^T_t\\
 Z^T_t \\
\end{array}
\right)   =    & \, \left(
\begin{array}{l}
\Lambda^T_0 \\
 Z^T_0 \\
\end{array}
\right)  -  \int_{0}^{t}\psi'(u_{T-s}(\infty))
 \left( \begin{array}{l}
\Lambda^T_{s-} \\
 0 \\
\end{array}
\right) \d s
                      + \sqrt{2\beta}  \int_{0}^{t}\int_{0}^{\Lambda^T_{s-}} \left( \begin{array}{l}
1 \\
 0 \\
\end{array}
\right)
 W(\d s,\d u) \notag\\
 &  + \int_{0}^{t}\int_{0}^{\infty}\int_{0}^{ \Lambda^T_{s-}} \left( \begin{array}{l}
r  \\
 0 \\
\end{array}
\right)
  \tilde{\texttt N}_T^{0}(\d s, \d r, {\rm d}\nu)  \notag \\
                      &\,   + \int_{0}^{t}\int_{0}^{\infty}\int_{1}^{Z^T_{s -}}  \left( \begin{array}{l}
r \\
 0 \\
\end{array}
\right){\texttt N}_T^1(\d s,\d r,\d{j})
\notag\\                     &\,  
      + \int_{0}^{t}\int_{0}^{\infty}\int_{0}^{\infty}\int_{1}^{Z^T_{s- }} \left( \begin{array}{l}
\quad r \\
 k-1 \\
\end{array}
\right) {\texttt N}_T^{2}(\d s,\d r,\d{k},\d{j})
\notag\\                     &\,  
   + 2\beta  \int_{0}^{t}   \left( \begin{array}{l}
Z^T_{s-} \\
 0 \\
\end{array}
\right) \d s, \qquad 0\leq t<T.                     
                          \label{TcoupledSDE}
    \end{align}}
    The equation  \eqref{TcoupledSDE} has a unique strong solution for arbitrary ($\mathcal{F}_0^T$-measurable) initial values $\Lambda_0^T\geq 0$ and $Z_0^T\in \mathbb{N}_0$ (where $\mathcal{F}_t^T:=\sigma((\Lambda_s^T,Z_s^T):s\leq t),t<T$).
Furthermore, under the assumption that $Z_0^T$ is an independent random variable which is Poisson distributed with intensity $u_T(\infty) \Lambda^T_0$ this unique solution satisfies the following: 
\begin{description}
\item[(i)] For $T>t\geq0$, conditional on $\mathcal{F}^{\Lambda^T}_t: = \sigma(\Lambda^T_s: s\leq t)$,  $Z^T_t $ is Poisson distributed with intensity $u_{T-t}(\infty)\Lambda^T_t$;
\item[(ii)] 
The process $(\Lambda^T_t, 0\leq t<T)$ is Markovian and a weak solution to \eqref{eqflow};  

\item[(iii)] Conditional on $\{Z^T_0=0\}$, the process $(\Lambda^T_t, 0\leq t<T)$ corresponds to a weak solution to \eqref{eqflow} conditioned to become extinct by  time $T$.
\end{description}

\end{theorem}

The SDE evolution in Theorem \ref{T}  mimics the skeletal decomposition in (\ref{coupledSDE}), albeit that  the different components in the decomposition are time-dependent. Putting the SDE representation aside, such time-varying skeletons have been observed in e.g. \cite{EW, DLG}.
We note that the underlying skeleton $Z^T$ can be thought of as a time-inhomogenous Galton--Watson process (a $T$-prolific skeleton) such that, at time $s< T$, its branching rate is given by  
\begin{equation}
q^{T-s}:= \frac{u_{T-s}(\infty)\psi'(u_{T-s}(\infty))-\psi(u_{T-s}(\infty))}{u_{T-s}(\infty)}
\label{rate}
\end{equation}
and offspring distribution is given by $\{p^{T-s}_k: k\geq 0\}$. This has the feature that the branching rate explodes towards the time horizon $T$. To see why, we can appeal to  \eqref{semigroup}, and note that
\[
\mathbb{P}_x[X_t = 0] = {\rm e}^{-u_t(\infty)x}, \qquad x,t>0,
\]
and hence $\lim_{t\to0}u_t(\infty) = \infty$. Moreover, one easily verifies from (\ref{mechanism}) that 
$
\lim_{\lambda\to\infty}[\lambda\psi'(\lambda) - \psi(\lambda)]/\lambda = \infty.
$
 Together, these facts imply the explosion of \eqref{rate} as $s\to T$.  
 
 We also note from the integrals involving ${\texttt N}_T^1$ and ${\texttt N}_T^{2}$ that there is mass immigrating off the space-time  trajectory of $Z^T$. Moreover, once mass has immigrated, the first four terms of \eqref{TcoupledSDE} show that 
it evolves as a time-inhomogenous CSBP.

{\color{black} Note, that in the supercritical setting $u_{T-t}(\infty)$ converges to $\lambda^*$ for all $t>0$ as $T\rightarrow\infty$.
This intuitively means that when $T$ goes to $\infty$, one can recover the prolific skeleton decomposition of Theorem \ref{coupled} from the time-inhomogeneous one of Theorem \ref{T}.}

\medskip

Finally with the finite-horizon SDE skeletal decomposition in Theorem \ref{T}, we may now turn our attention to understanding what happens when we observe the solution {\color{black} to \eqref{TcoupledSDE} in the (sub)critical case } on a finite time horizon $[0,t_0]$, {\color{black} and} we condition on there being at least one $T$-prolific genealogy,  {\color{black} while letting } $T\rightarrow\infty$.

\begin{theorem}\label{skeletontospine} 
Suppose that $\psi$ is a critical or subcritical branching mechanism  such that Grey's condition \eqref{grey}  holds. Suppose, moreover, that $((\Lambda^T_t, Z^T_t), 0\leq t<T)$ is a weak solution to  \eqref{TcoupledSDE} and that $Z_0^T$ is an independent random variable which is Poisson distributed with intensity $u_T(\infty) \Lambda^T_0$. 
Then, conditional on the event  $Z_0^T>0$, in the sense of weak convergence with respect to the Skorokhod topology on  $\mathbb{D}([0,\infty),\mathbb{R}^2)$,  for all $t_0>0$,
\[
((\Lambda^T_t, Z^T_t), 0\leq t\leq t_0) \rightarrow ((X^\uparrow_t, 1), 0\leq t\leq  t_0),
\]
as $T\to\infty$,
where $X^\uparrow$  is a weak solution to \eqref{surv}.
\end{theorem}

Theorem \ref{skeletontospine} puts the phenomena of spines and skeletons in the same framework. Roughly speaking, any subcritical branching population contains a naturally embedded skeleton which describes the `fittest' genealogies. In our setting `fittest' means surviving until time $T$ but other notions of fitness can be considered, especially when one introduces a spatial type to mass in the branching process. For example in \cite{HHK} a  branching Brownian motion in a strip is considered, where `fittest genealogies' pertains to those lines of descent which survive in the strip for all eternity. Having at least one line of descent in the skeleton corresponds to the event of survival. Thus, conditioning on survival as we make the survival event itself increasingly unlikely, e.g. by taking $T\to\infty$  in our model or taking the width of the strip down to a critical value in the branching Brownian motion model, the natural stochastic behaviour of the skeleton is to thin down to a single line of decent. This phenomenon was originally observed in  \cite{EW}, where the  scaling limit of a Galton--Watson processes conditioned on survival is shown to converge to the immortal particle decomposition of the $(1+\beta)$-superprocess conditioned on survival. 

\bigskip

The remainder of the paper is structured as followed. In the next section we explain the heuristic behind how \eqref{eqflow} can be decoupled into components that arise in \eqref{coupledSDE}. The heuristic is used in Section \ref{proof1} where the proof of Theorem \ref{coupled} is given. In this sense our proof of Theorem \ref{coupled} has the feel of a `guess-and-verify' approach.
In Section \ref{explore}, again in the spirit of a `guess-and-verify' approach, we use ideas from the classical description of the exploration process of CSBPs in e.g. \cite{DLG} to provide the heuristic behind the mathematical structures that lie behind the proof of Theorem \ref{T}. Given the similarity of this proof to that of Theorem \ref{coupled}, it is sketched in Section \ref{proof2}. Finally in Section \ref{proof3} we provide the proof of Theorem \ref{skeletontospine}.

\section{Thinning of the CSBP SDE}\label{marks}

In this section, we will perform an initial manipulation of the SDE \eqref{eqflow}, which we will need in order to make comparative statements for Theorems \ref{coupled} and \ref{T}.
 To this end, we will introduce some independent marks on the atoms of the Poisson process $N$  driving \eqref{eqflow} and use them to thin out various contributions to the SDE evolution.

Denote by $(t_i,r_i,\nu_i: i\in \NN)$  some enumeration of the atoms of $N$ and recall that  $\mathbb{N}_0 = \{0\}\cup\mathbb{N}$.
By enlarging the probability space, we can introduce an additional mark to atoms of $N$, say  $({k}_i: i\in \NN)$, resulting in an  `extended' Poisson random measure,
 \begin{equation}\label{extendedPoisson}
 {\cal N}({\rm d}s,{\rm d}r,{\rm d}\nu,{\rm d}k):=\sum_{i\in \NN}  \delta_{(t_i,r_i,\nu_i, k_i)} ({\rm d}s,{\rm d}r,{\rm d}\nu,{\rm d}k)
 \end{equation} on $[0,\infty)^3\times\mathbb{N}_0$ 
 with intensity
 \[
 {\rm d}s \otimes \Pi({\rm d}r)\otimes {\rm d}\nu\otimes
 \frac{(\lambda r)^k}{k!} {\rm e}^{-\lambda r}\sharp({\rm d}k).
 \]
Now define three random measures by 
 \[
   N^{0}({\rm d}s,{\rm d}r,{\rm d}\nu)={\cal N}({\rm d}s,{\rm d}r,{\rm d}\nu,\{ 0  \}),
   \]
 \[
   N^{1}({\rm d}s,{\rm d}r,{\rm d}\nu)= {\cal N}({\rm d}s,{\rm d}r,{\rm d}\nu,\{1 \})
   \]
   and
\[
N^{2 }({\rm d}s,{\rm d}r,{\rm d}\nu)= {\cal N}({\rm d}s,{\rm d}r,{\rm d}\nu,\{ k \geq2  \}) .
\]
Classical Poisson thinning now tells us that 
$  N^{0}$, $  N^{1}$  and $  N^{2}$ are independent  Poisson point processes on $[0, \infty)^{3}$ with respective intensities ${\rm d}s\otimes {\rm e}^{-{\lambda} r} \Pi({\rm d}r) \otimes {\rm d}\nu$, ${\rm d}s\otimes ({\lambda}r)  {\rm e}^{-{\lambda} r} \Pi({\rm d}r) \otimes {\rm d}\nu$ and   
$
{\rm d}s\otimes\sum_{k=2}^{\infty} ({\lambda}r)^k {\rm e}^{-{\lambda} r} \Pi({\rm d}r)/{k!}\otimes {\rm d}\nu.
$

 With these thinned Poisson random measures in hand, we may start to separate out the different stochastic integrals in \eqref{eqflow}.  We have that, for $t\geq 0$,
 
     \begin{align}
    X_t= 
                       x &+ \alpha \int_{0}^{t} X_{s- }{\rm d}s
                      + \sqrt{2\beta}  \int_{0}^{t}\int_{0}^{X_{s -}} W({\rm d}s,{\rm d}u)  + \int_{0}^{t}\int_{0}^{\infty}\int_{0}^{X_{s -}}r \tilde{N}^{0}({\rm d}s, {\rm d}r, {\rm d}\nu)  \notag \\
                      &\,   + \int_{0}^{t}\int_{0}^{\infty}\int_{0}^{X_{s -}}r N^{1}({\rm d}s, {\rm d}r, {\rm d}\nu) 
                        + \int_{0}^{t}\int_{0}^{\infty}\int_{0}^{X_{s- }}r N^{2}({\rm d}s, {\rm d}r, {\rm d}\nu) \notag\\
                      &\, -\int_{0}^{t}\int_{0}^{\infty}   X_{s -} \sum_{n=1}^{\infty} \frac{({\lambda}r)^n}{n!}  {\rm e}^{-{\lambda} r} r  \Pi({\rm d}r)  {\rm d}s \notag\\
=     x& - \psi'(\lambda) \int_{0}^{t} X_{s } {\rm d}s
                      + \sqrt{2\beta}  \int_{0}^{t}\int_{0}^{X_{s -}} W({\rm d}s,{\rm d}u)  + \int_{0}^{t}\int_{0}^{\infty}\int_{0}^{X_{s -} }r \tilde{N}^{0}({\rm d}s, {\rm d}r, {\rm d}\nu)   \notag\\
                      &\,   + \int_{0}^{t}\int_{0}^{\infty}\int_{0}^{X_{s -}}r N^{1}({\rm d}s, {\rm d}r, {\rm d}\nu) 
                        + 2\beta {\lambda} \int_{0}^{t} X_{s -} {\rm d}s  \notag\\                       
                        &\,   + \int_{0}^{t}\int_{0}^{\infty}\int_{0}^{X_{s- }}r N^{2}({\rm d}s, {\rm d}r, {\rm d}\nu),
                        \label{Xsdemarked}
    \end{align}
   where in the last equality we have used the easily derived fact that $-\int_{(0,\infty)}(1-{\rm e}^{-{\lambda} r})r\Pi({\rm d}r)= -\alpha +2\beta {\lambda}-  \psi'(\lambda) $.
 Recalling \eqref{psicomp}, 
 the  first line  in the last equality   of \eqref{Xsdemarked} corresponds to the  dynamics of a subcritical CSBP with branching mechanism $\psi_\lambda$.  
 
Inspecting the statement of Theorem \ref{coupled}, we see intuitively that in order to prove this result, our job  is to show that the integrals on the right-hand side  of \eqref{Xsdemarked} driven by $N^1$ and $N^2$ can be identified with the mass that immigrates off the skeleton. 

\section{$\lambda$-Skeleton: Proof of Theorem \ref{coupled}} \label{proof1}

We start by addressing the claim that \eqref{coupledSDE} possesses a unique strong solution. Thereafter we prove claims (i) and (ii) of the theorem in order. 

\medskip

We can identify the existence of any  weak solution to \eqref{coupledSDE} with initial value $(\Lambda_0, Z_0)= (x,n)$, $x\geq0$, $n\in\mathbb{N}_0$, by introducing additionally marked versions of the Poisson random measures ${\texttt N}^{1}$ and ${\texttt N}^{2}$, as well as an additional Poisson random measure ${\texttt N}^*$. We will insist that ${\texttt N}^{1}(\d s,\d r,\d{j}, \d\omega)$  has intensity $\d s\otimes r{\rm e}^{-\lambda r}\Pi(\d r)\otimes \sharp(\d j)\otimes\mathbb{P}^{(\lambda)}_r(\d\omega)$ on $[0,\infty)^2\times\mathbb{N}_0\times\mathbb{D}([0,\infty), \mathbb{R})$, ${\texttt N}^{2}(\d s,\d r,\d{k},\d{j}, \d\omega)$ has intensity $ \psi'(\lambda)\d s\otimes \eta_k(\d r)  \otimes p_k\sharp(\d k)\otimes \sharp(\d{j})\otimes \mathbb{P}^{(\lambda)}_r(\d\omega)$ on $[0,\infty)^2\times\mathbb{N}_0\times\mathbb{N}_0\times\mathbb{D}([0,\infty), \mathbb{R})$ and 
${\texttt N}^*(\d s, \d{j}, \d \omega)$ has intensity $2\beta\d s\otimes \sharp(\d j)\otimes \mathbb{Q}^{(\lambda)}(\d\omega)$
on $[0,\infty)\times\mathbb{N}_0\times\mathbb{D}([0,\infty), \mathbb{R})$, where $\mathbb{P}^{(\lambda)}_r$ is the law of a $\psi_\lambda$-CSBP with initial value $r\geq 0$ (formally speaking $\mathbb{P}^{(\lambda)}_0$ is the law of the null process) and $\mathbb{Q}^{(\lambda)}$ is the associated excursion measure.

Our proposed solution to  to \eqref{coupledSDE} will be to first define $(Z_t, t\geq 0)$ as the continuous-time Galton--Watson process with branching rate $\psi'(\lambda)$ and
offspring distribution given by $(p_k, k\geq 0)$. It is then easy to otherwise $Z$ in the more complicated form 
\[
Z_t = 
 \int_{0}^{t}\int_{0}^{\infty}\int_{0}^{\infty}\int_{1}^{Z_{s- }} \int_{\mathbb{D}([0,\infty), \mathbb{R})}(k-1) \, {\texttt N}^{2}(\d s,\d r,\d{k},\d{j}, \d\omega), \qquad t\geq 0.
\] 
 Next we take 
\begin{equation}
\Lambda_t =X^{(\lambda)}_t + D_t,\qquad t\geq 0,
\label{weaksolutionrepresentation}
\end{equation}
where $X^{(\lambda)}$ is an autonomously independent copy of a $\psi_\lambda$-CSBP issued with initial mass $x$ and, given   ${\texttt N}^{1}$ and  ${\texttt N}^{2}$,  $D_t$, $t\geq 0$ is the uniquely identified (up to almost sure modification) `dressed skeleton' described by 
\begin{align*}
D_t&=  \int_{0}^{t}\int_{0}^{\infty}\int_{1}^{Z_{s -}}\int_{\mathbb{D}([0,\infty), \mathbb{R})}  
\omega_{t-s} \, {\texttt N}^1(\d s,\d r,\d{j}, \d\omega)\\
&\hspace{2cm}
      + \int_{0}^{t}\int_{0}^{\infty}\int_{0}^{\infty}\int_{1}^{Z_{s- }} \int_{\mathbb{D}([0,\infty), \mathbb{R})} \omega_{t-s} \, {\texttt N}^{2}(\d s,\d r,\d{k},\d{j}, \d\omega)\\
      &
   \hspace{4cm}+  \int_0^t \int_{1}^{Z_{s- }}\int_{\mathbb{D}([0,\infty), \mathbb{R})}  \omega_{t-s} \, {\texttt N}^*(\d s, \d{j}, \d \omega).
\end{align*}
To see why this provides a weak solution to \eqref{coupledSDE}, we may appeal to the Martingale representation  the  Markov pair $(\Lambda, Z)$ described above. In particular   the  generator of $(\Lambda, Z)$ can be identified consistently with the generator of the process associated to \eqref{coupledSDE}; that it to say, their common generator is given by 
\begin{align*}
\mathbf{L} f(x, n) &= - \psi'(\lambda)x \frac{\partial f}{\partial x } (x, {n}) + \beta x\frac{\partial^2f}{\partial x^2} (x, {n})     \\
    &+ x\int_0^\infty  [f(x+ r, {n}) -f(x, {n} ) - 
    r\frac{\partial f}{\partial x}(x, {n}) ]{\rm e}^{-r \lambda}\Pi(\d r)
    \\
    &+{n}\int_0^\infty \sum_{j\geq 1} [f(x+ r, {n} +j-1) -f(x, {n}) ]
     \frac{r^{j} \lambda^{j-1}}{j!}{\rm e}^{-r \lambda}\Pi(\d r) 
\\
   & +\beta \lambda {n}[f(x , {n} +1) -f(x, {n})]+ \frac{\psi(
   \lambda)}{\lambda} n[f(x , {n} -1) -f(x, {n})]+ 2\beta {n} \frac{\partial f}{\partial x}(x, {n}),
\end{align*}
for $x\geq 0$, $n\in\mathbb{N}_0$, and for all non-negative, smooth and compactly supported functions $f$. (Here,  the penultimate term is understood to be zero when $n=0$.) With this sense of commonality for their generators, it is then easy to verify the conditions of  Theorem 2.3 of \cite{Kurtz} and thus  to conclude that $(\Lambda, Z)$ provides a weak solution to \eqref{coupledSDE}.

\medskip

Pathwise uniqueness is also relatively  easy to establish. Indeed, suppose that $\Lambda$ is the first component of  any path solution to \eqref{coupledSDE} with driving source of randomness ${\texttt N}^{0}$, ${\texttt N}^{1}$, ${\texttt N}^{2}$ and $W$ and suppose that we write it in the form 
\begin{align*}
\Lambda_t 
 =: \Lambda_0 + H( \Lambda_{s}, s<t )+I_t \qquad t\geq 0,   
\end{align*}
where
\begin{align*}
H( \Lambda_{s}, s<t ) &=  -\psi'(\lambda) \int_{0}^{t}
\Lambda_{s-}  \d s+ \sqrt{2\beta}  \int_{0}^{t}\int_{0}^{\Lambda_{s-}}  W(\d s,\d u)\\
&   \hspace{3cm}+ \int_{0}^{t}\int_{0}^{\infty}\int_{0}^{ \Lambda_{s-}} r    \tilde{{\texttt N}}^{0}(\d s, \d r, {\rm d}\nu),\qquad t\geq0,
\end{align*}
and
\begin{align*}
I_t &=   \int_{0}^{t}\int_{0}^{\infty}\int_{1}^{Z_{s -}} 
r {\texttt N}^1(\d s,\d r,\d{j})\\
&      \hspace{2cm}+ \int_{0}^{t}\int_{0}^{\infty}\int_{0}^{\infty}\int_{1}^{Z_{s- }} r  {\texttt N}^{2}(\d s,\d r,\d{k},\d{j}) 
   + 2\beta  \int_{0}^{t}   Z_{s-}  \d s, \qquad t\geq 0.    
\end{align*}
Recalling that the almost sure path of  $Z$ is uniquely defined  by $ {\texttt N}^{2}$, it follows that, if $\Lambda^{(1)}$ and $\Lambda^{(2)}$ are two path solutions to \eqref{coupledSDE} with the same initial value, then 
\[
\Lambda^{(1)}_t-\Lambda^{(2)}_t = H( \Lambda^{(1)}_{s}, s<t )- H( \Lambda^{(2)}_{s}, s<t ), \qquad t\geq 0.
\]
The reader will now note that the above equation is precisely the  SDE one obtains when looking at the path difference between two solutions of an SDE of the type given in \eqref{eqflow}. Since there is pathwise uniqueness for \eqref{eqflow}, we easily conclude that $\Lambda^{(1)} = \Lambda^{(2)}$ almost surely.

Finally, taking account of the existence of a weak solution and pathwise uniqueness, we may appeal to an appropriate version of the Yamada-Watanabe Theorem, see for example Theorem 1.2 of \cite{BLG}, to deduce that \eqref{coupledSDE}  possesses a unique strong solution.
And since this holds for every fixed initial configuration $x$ and $n$, it also holds when the initial values are independently randomised.

\medskip

\noindent {\bf (i)} This claim requires an analytical verification and, in some sense, is similar in spirit to the  proof that, for $t\geq 0$, $Z_t|\Lambda_t$ is Poisson distributed with rate $\lambda^*\Lambda_t$ in the prolific skeletal decomposition found in \cite{BKM}. A fundamental difference here is that we work with SDEs, and hence stochastic calculus, rather than integral equations for semigroups as in \cite{BKM} is needed and, moreover, the parameter $\lambda$ need not be the minimal value, $\lambda^*$, in its range.

Standard arguments show that the solution to \eqref{coupledSDE} is a strong Markov process and accordingly we write $\mathbf{P}_{x,n}$, $x>0$, $n\in\mathbb{N}_0$ for its probabilities. Moreover, with an abuse of notation we write, for $x>0$, 
\begin{equation}
\mathbf{P}_x(\cdot) = \sum_{n\geq 0}\frac{(\lambda x)^n}{n!}{\rm e}^{-\lambda x}\mathbf{P}_{x,n}(\cdot).
\label{abuse}
\end{equation}
 Define  $f_t(\eta, \theta):=\mathbf{E}_x[{\rm e}^{-\eta \Lambda_t-\theta Z_t}]$, $x,\theta,\eta, t\geq 0$, and let
 $F_t:= {\rm e}^{-\eta \Lambda_t-\theta Z_t}$, $t\geq 0$.  Using  It\^{o}'s formula for semi-martingales, cf. Theorem 32 of \cite{P}, for $t\geq 0$,
\begin{equation*}
\begin{split}
{\rm d}F_t=&-\eta F_{t-}{\rm d}\Lambda_t-\theta F_{t-}{\rm d}Z_t+\frac{1}{2}\eta^2 F_{t-}{\rm d}[\Lambda,\Lambda]_t^c+\frac{1}{2}\theta^2 F_{t-}{\rm d}[Z,Z]_t^c\\
&+\eta \theta F_{t-}{\rm d}[\Lambda,Z]_t^c+\Delta F_t+\eta F_{t-}\Delta \Lambda_t+\theta F_{t-}\Delta Z_t.
\end{split}
\end{equation*}
(Here and throughout the remainder of this paper, for any stochastic process $Y$, we use the notation that $\Delta Y_t = Y_t-Y_{t-}$.) As $Z$ is a pure jump process, we have that $[Z,Z]_t^c=[\Lambda,Z]_t^c=0$.
Taking advantage of the fact that  
\[
F_t=F_{t-}{\rm e}^{-\eta\Delta \Lambda_t-\theta\Delta Z_t},
\] 
we may thus write in integral form
\begin{align*}
F_t =&F_0-\eta\int_0^t F_{s-}{\rm d}\Lambda_s-\theta\int_0^t F_{s-}{\rm d}Z_s+\beta\eta^2\int_0^t F_{s-}\Lambda_s {\rm d}s\\
&+\sum_{s\leq t}F_{s-}\left\lbrace {\rm e}^{-\eta\Delta \Lambda_s-\theta \Delta Z_s}-1+\eta\Delta \Lambda_s+\theta \Delta Z_s \right\rbrace,
\end{align*}
where the sum is taken over the countable set of discontinuities of $(\Lambda, Z)$.
We can split up the sum of discontinuities according to the Poisson random measure in \eqref{coupledSDE} that
is responsible for the discontinuity. Hence, writing $\Delta^{(j)}$, $j=0,1,2$, to mean an increment coming from each of the three Poisson random measures, 
\begin{align}\label{Fint}
F_t &=F_0-\eta\int_0^t F_{s-}{\rm d}\Lambda_s-\theta\int_0^t F_{s-}{\rm d}Z_s+\beta\eta^2\int_0^t F_{s-}\Lambda_{s-} {\rm d}s\notag\\
&+\sum_{s\leq t}F_{s-}\left\lbrace {\rm e}^{-\eta\Delta^{(0)} \Lambda_s}-1+\eta\Delta^{(0)} \Lambda_s \right\rbrace\notag\\
&+\sum_{s\leq t}F_{s-}\left\lbrace {\rm e}^{-\eta\Delta^{(1)} \Lambda_s}-1+\eta\Delta^{(1)} \Lambda_s \right\rbrace\notag\\
&+\sum_{s\leq t}F_{s-}\left\lbrace {\rm e}^{-\eta\Delta^{(2)} \Lambda_s-\theta \Delta Z_s}-1+\eta\Delta^{(2)} \Lambda_s+\theta \Delta Z_s \right\rbrace.
\end{align}

Now, note that we can re-write  the first element of the vectorial SDE \eqref{coupledSDE} as
\begin{align*}
\Lambda_t=  &   \Lambda_0-\psi'(\lambda)\int_0^t \Lambda_{s-} {\rm d}s+\sqrt{2\beta}\int_0^t\int_0^{\Lambda_{s-}} W({\rm d}s, {\rm d}u)+M_t\\
&+\sum_{s\leq t}\Delta^{(1)} \Lambda_s+\sum_{s\leq t}\Delta^{(2)} \Lambda_s+2\beta\int_0^t Z_{s-}{\rm d}s,\qquad t\geq 0,
\end{align*}
where $M_t$ is a zero-mean martingale corresponding to the integral   in \eqref{coupledSDE}   with respect to $\tilde{\texttt N}^0$.
Therefore performing the necessary calculus in (\ref{Fint}) for the integral with respect to  ${\rm d}\Lambda_t$, we get that
\begin{align*}
F_t & -F_0-\eta(\psi'(\lambda)+\eta\beta)\int_0^t F_{s-}\Lambda_{s-} {\rm d}s - \sum_{s\leq t}F_{s-}\left\lbrace {\rm e}^{-\eta\Delta^{(2)} \Lambda_s-\theta \Delta Z_s}-1 \right\rbrace\\
&+2\eta\beta\int_0^t F_{s-}Z_{s-} {\rm d}s-\sum_{s\leq t}F_{s-}\left\lbrace {\rm e}^{-\eta\Delta^{(1)} \Lambda_s}-1 \right\rbrace - \sum_{s\leq t}F_{s-}\left\lbrace {\rm e}^{-\eta\Delta^{(0)} \Lambda_s}-1+\eta\Delta^{(0)} \Lambda_s \right\rbrace, \end{align*}
for $t\geq 0$,
is equal to a zero-mean martingale which is the sum of the previously mentioned $M_t$, $t\geq 0$,  and the white noise integral.
Taking expectations, we thus have 
\begin{equation*}
\begin{split}
f_t(\eta,\theta)=&f_0(\eta,\theta)+\eta(\psi'(\lambda)+\eta\beta)\mathbf{E}_x\int_{0}^{t}[F_{s-}\Lambda_{s-}]{\rm d}s-2\eta\beta \mathbf{E}_x\int_{0}^{t}[F_{s-}Z_{s-}]{\rm d}s\\
&+\mathbf{E}_x \int_0^t[ F_{s-} \Lambda_{s-}] {\rm d}s\left(\int_0^\infty ({\rm e}^{-\eta r}-1+\eta r){\rm e}^{-\lambda r}\Pi({\rm d}r)\right)\\
&+\mathbf{E}_x \int_0^t [F_{s-} Z_{s-}] {\rm d}s\left( \int_0^\infty ({\rm e}^{-\eta r}-1)r {\rm e}^{-\lambda r}\Pi({\rm d}r)\right)\\
&+\mathbf{E}_x \int_0^t \frac{1}{\lambda}[F_{s-}Z_{s-}]{\rm d}s\Bigg( \sum_{k=0, k\neq 1}^{\infty} \int_0^\infty ({\rm e}^{-\eta r -\theta(k-1)}-1)\Big\{ \psi(\lambda)\mathbf{1}_{\{k=0\}}\delta_0({\rm d}r)\\
&\hspace{4.5cm}+\delta_0({\rm d}r)\beta\lambda^2\mathbf{1}_{\{k=2\}}+\mathbf{1}_{\{k\geq 2\}}\frac{(\lambda r)^k}{k!}{\rm e}^{-\lambda r}\Pi({\rm d}r)\Big\} \Bigg) .
\end{split}
\end{equation*}
Accumulating terms, we find that  $f_t(\eta,\theta)$ satisfies the following PDE
\begin{align}
\frac{\partial}{\partial t}f_t(\eta,\theta) &=A_\lambda(\eta,\theta)\frac{\partial}{\partial \eta}f_t(\eta,\theta)+B_\lambda(\eta,\theta)\frac{\partial}{\partial \theta}f_t(\eta,\theta),\notag\\
f_0(\eta,\theta)&={\rm e}^{-(\eta +\lambda (1-{\rm e}^{-\theta})) x },
\label{QL}
\end{align}
where 
\begin{align*}
A_\lambda(\eta,\theta)&=\eta(-\psi'(\lambda) -\eta\beta)-\int_0^\infty ({\rm e}^{-\eta r}-1+\eta r){\rm e}^{-\lambda r}\Pi({\rm d}r)\\
B_\lambda(\eta,\theta)&=2\eta\beta-\sum_{k=0}^\infty \int_0^\infty ({\rm e}^{-\eta r-\theta(k-1)}-1)\Bigg\{ \frac{\psi(\lambda)}{\lambda}\mathbf{1}_{\{k=0\}}\delta_0({\rm d}r)+\delta_0({\rm d}r)\beta\lambda\mathbf{1}_{\{k=2\}}\\
&\hspace{10cm}+\mathbf{1}_{\{k\geq 1\}}\frac{\lambda^{k-1} r^k}{k!}{\rm e}^{-\lambda r}\Pi({\rm d}r)\Bigg\}.
\end{align*}

{\color{black}
Standard theory for linear partial differential equation \eqref{QL}, see for example Chapter 3 (Theorem 2, p107) of \cite{evans} and references therein, tells us that it has a unique local solution. Our aim now is to show that this solution is also represented by 
\begin{equation}\label{cblaplace}
\mathbb{E}_x[{\rm e}^{-(\eta+\lambda(1-{\rm e}^{-\theta})){X}_t}] = {\rm e}^{- u_t(\eta + \lambda (1- {\rm e}^{-\theta}))x }, \qquad x,t, \theta, \eta\geq 0,
\end{equation} where we recall that $X$ is the $\psi$-CSBP. 
To this end, let us define $\kappa = \eta + \lambda (1- {\rm e}^{-\theta})$ and note that, for $x,t, \kappa \geq 0$, 
$
g_t(\kappa):  =\mathbb{E}_x[\exp\{-\kappa X_t\}]
$
satisfies
%
%
%
%
%
\begin{equation}\label{Ypde}
\begin{split}
\frac{\partial}{\partial t}{g}_t({\kappa}) &=-\psi(\kappa) \frac{\partial}{\partial {\kappa}}{g}_t({\kappa}),\\
{g}_0({\kappa})&={\rm e}^{-{\kappa} x}.
\end{split}
\end{equation}
See for example Exercise 12.2 in \cite{Kbook}.
After a laborious amount of algebra one can verify that   $-\psi({\kappa})=A_\lambda(\eta,\theta)+\lambda {\rm e}^{-\theta}B_\lambda(\eta,\theta)$ and hence we may develop the right hand side of \eqref{Ypde} and write, for $x, t, \eta, \theta\geq 0$,
\[
\frac{\partial}{\partial t}{g}_t({\kappa}) =A_\lambda(\eta,\theta)\frac{\partial}{\partial {\kappa}}{g}_t({\kappa})+\lambda {\rm e}^{-\theta}B_\lambda(\eta,\theta) \frac{\partial}{\partial {\kappa}}{g}_t({\kappa})= A_\lambda(\eta,\theta)\frac{\partial}{\partial \eta}{g}_t({\kappa})+B_\lambda(\eta,\theta) \frac{\partial}{\partial {\theta}}{g}_t({\kappa}).
\] 

Now we choose $x=1$. Then local uniqueness of the solution to \eqref{QL} (or equivalently the local uniqueness of \eqref{Ypde}) thus tells us that there exists $t_0>0$ such that $g_t( \eta + \lambda (1- {\rm e}^{-\theta})) = f_t (\eta, \theta)$ for all $\eta, \theta \geq 0$ and $t\in [0,t_0]$.

\medskip

In conclusion, now that we have proved that for  $t\in[0,t_0]$ 
\begin{equation}
\mathbb{E}_1[{\rm e}^{-(\eta+\lambda(1-{\rm e}^{-\theta})){X}_t}] = \mathbf{E}_1[{\rm e}^{-\eta \Lambda_t-\theta Z_t}], \qquad x, \theta, \eta\geq 0,
\label{doitforT}
\end{equation}
we can sequentially observe the following implications. Firstly, setting $\theta = 0$ and $\eta>0$, we see that $\Lambda_t$ under $\mathbf{P}_1$ has the same distribution as $X_t$ under $\mathbb{P}_1$ for all $t\in[0,t_0]$. Next, setting both $\eta,\theta>0$, we observe that,  $(\Lambda_t, Z_t)$ under $\mathbf{P}_1$ has the same law as $(X_t, \texttt{Po}(\lambda x)|_{x=X_t})$ under $\mathbb{P}_1$, where $\texttt{Po}(\lambda x)$ is an autonomously  independent Poisson random variable with rate $\lambda x$. In particular, it follows that, for all $t\in[0,t_0]$, under $\mathbf{P}_1$, the law of $Z_t$ given $\Lambda_t$ is $\texttt{Po}(\lambda \Lambda_t)$. 

To get a global result we first show that the previous conclusions hold for any initial mass $x>0$ on the time-interval $[0,t_0]$, then using Markov property, we extend the results for any $t>0$.
First, from \eqref{cblaplace} we can observe that
\[
\mathbb{E}_x\left[ \mathrm{e}^{-(\eta+\lambda(1-\mathrm{e}^{-\theta}))X_t}\right]=\left(\mathbb{E}_1\left[ \mathrm{e}^{-(\eta+\lambda(1-\mathrm{e}^{-\theta}))X_t}\right]\right)^x.
\] 
Thus in order to extend the previous results to any $x>0$ we only need to prove that 
\[
\mathbf{E}_x\left[ \mathrm{e}^{-\eta\Lambda_t-\theta Z_t}\right]=\left(\mathbf{E}_1\left[ \mathrm{e}^{-\eta\Lambda_t-\theta Z_t}\right]\right)^x.
\]
Recalling the representation \eqref{weaksolutionrepresentation} and the notation \eqref{abuse} we can write, for $t\leq t_0$,
\begin{equation*}
\begin{split}
\mathbf{E}_x\left[ \mathrm{e}^{-\eta\Lambda_t-\theta Z_t}\right]&=\sum_{n\geq 0}\frac{(\lambda x)^n}{n!}\mathrm{e}^{-\lambda x}\mathbf{E}_{(x,n)}\left[ \mathrm{e}^{-\eta\Lambda_t-\theta Z_t}\right]\\
&=\sum_{n\geq 0}\frac{(\lambda x)^n}{n!}\mathrm{e}^{-\lambda x}\mathbf{E}_{x}\left[ \mathrm{e}^{-\eta X_t^{(\lambda)}}\right]\mathbf{E}_{(0,n)}\left[ \mathrm{e}^{-\eta D_t-\theta Z_t}\right]\\
&=\left(\mathbf{E}_1\left[ \mathrm{e}^{-\eta X_t^{(\lambda)}}\right]\right)^x\sum_{n\geq 0}\frac{(\lambda x)^n}{n!}\mathrm{e}^{-\lambda x}\left(\mathbf{E}_{(0,1)}\left[ \mathrm{e}^{-\eta D_t-\theta Z_t}\right]\right)^n\\
&=\left(\mathbf{E}_1\left[ \mathrm{e}^{-\eta X_t^{(\lambda)}}\right]\mathrm{e}^{\lambda\left( \mathbf{E}_{(0,1)}\left[ \mathrm{e}^{-\eta D_t-\theta Z_t}\right]-1\right)}\right)^x\\
&=\left(\mathbf{E}_1\left[ \mathrm{e}^{-\eta\Lambda_t-\theta Z_t}\right]\right)^x,
\end{split}
\end{equation*}
as required.

Now take $t_0<t\leq 2t_0$, and use the tower property to get
\begin{equation*}
\mathbf{E}_x\left[ \mathrm{e}^{-\eta \Lambda_t-\theta Z_t}\right]=\mathbf{E}_x\left[ \mathbf{E}_{\Lambda_{t_0}}\left[ \mathrm{e}^{-\eta\Lambda_{t-t_0}-\theta Z_{t-t_0}}\right]\right]=\int_{\mathbb{R}_+}\mathbf{E}_y\left[ \mathrm{e}^{-\eta\Lambda_{t-t_0}-\theta Z_{t-t_0}}\right]\mathbf{P}_x(\Lambda_{t_0}\in \mathrm{d}y),
\end{equation*} 
and similarly
\begin{equation*}
\mathbb{E}_x\left[ \mathrm{e}^{-(\eta+\lambda(1-\mathrm{e}^{-\theta}))X_t}\right]
=\int_{\mathbb{R}_+}\mathbb{E}_y\left[ \mathrm{e}^{-(\eta+\lambda(1-\mathrm{e}^{-\theta}))X_{t-t_0}}\right]\mathbb{P}_x(X_{t_0}\in \mathrm{d}y).
\end{equation*} 
Thus using local uniqueness and the previously deduced implications on $[0,t_0]$ we see that 
\[
\mathbf{E}_x\left[ \mathrm{e}^{-\eta \Lambda_t-\theta Z_t}\right]=\mathbb{E}_x\left[ \mathrm{e}^{-(\eta+\lambda(1-\mathrm{e}^{-\theta}))X_t}\right]
\] 
on $t\in [0,2 t_0]$, and by iterating the previous argument we get equality for any $t>0$.

Finally, }on account of the fact that  $(\Lambda_t, Z_t)$, $t\geq0$, is a joint Markovian pair, this {\color{black} now global} Poissonisation allows us to infer that $\Lambda_t$, $t\geq 0$, is itself Markovian. Indeed, for any bounded measurable and positive $h$ and $s,t\geq 0$,
\[
\mathbf{E}[h(\Lambda_{t+s})| \mathcal{F}^\Lambda_t]  = \sum_{n\geq 0} \frac{(\lambda \Lambda_t)^n}{n!} {\rm e}^{-\lambda \Lambda_t}\mathbf{E}_{x,n}[h(\Lambda_{s})]_{x = \Lambda_t} = \mathbf{E}_{x}[h(\Lambda_{s})]_{x = \Lambda_t}.
\]
We may now conclude that for all $t\geq 0$ and $x>0$, under $\mathbf{P}_x$, $Z_t|\mathcal{F}^\Lambda_t$ is $\texttt{Po}(\lambda \Lambda_t)$-distributed as required.

\medskip

\noindent{\bf (ii)} {\color{black} We have seen that the pair $((\Lambda_t,Z_t),t\geq 0)$ is a Markov process for any initial state $(x,n)$ but, due to the dependence on $Z$, on its own $(\Lambda_t,t\geq 0)$ is not Markovian.
However considering \eqref{doitforT} we see that after the Poissonisation of $Z_0$, $(\Lambda_t, t\geq 0)$ becomes a Markov process with semi-group that agrees with that of $(X_t, t\geq 0)$.}
On account of the fact that $X$ is the unique weak solution to \eqref{eqflow}, it automatically follows that $\Lambda$ also represents the unique weak solution to \eqref{eqflow}.
\medskip

\noindent{\bf (iii)} 
Since the event $\{Z_0=0\}$ implies the event $\{Z_t=0,t\geq 0\}$, the system \eqref{coupledSDE} reduces to the SDE
 \begin{align*}
    \Lambda_t= x& + \psi'(\lambda) \int_{0}^{t} \Lambda_{s-}{\rm d}s 
                      + \sqrt{2\beta}  \int_{0}^{t}\int_{0}^{\Lambda_{s-}} W({\rm d}s,{\rm d}u) 
                      + \int_{0}^{t}\int_{0}^{\infty}\int_{0}^{\Lambda_{s-}}r \tilde{N}^0({\rm d}s, {\rm d}r, {\rm d}\nu), 
    \end{align*}
  which has the exact form as the SDE describing the evolution of a CSBP with branching mechanism $\psi_\lambda$.   
 \hfill $\square$

\section{Exploration of subcritical CSBPs}\label{explore}
The objective of this section is to give a heuristic description of how the notion of a prolific skeleton emerges in the subcritical case and specifically why the structure of the SDE \eqref{TcoupledSDE} is meaningful in this respect. We need to be  careful about what one means by `prolific' but nonetheless, the inspiration for a decomposition  can be gleaned by examining in more detail the description of subcritical CSBPs through the exploration process. 

We  assume throughout the conditions of Theorem \ref{T}.  That is to say, $X$  a (sub)critical $\psi$-CSBP where $\psi$ satisfies Grey's condition \eqref{grey}.
Let $(\xi_t,t\geq 0)$ be a spectrally positive L\'evy process with Laplace exponent $\psi$.
Using the classical work of \cite{LeGLeJ, LeGLeJ1} (see also \cite{DG, DLG, LeG})  we can use generalised Ray--Knight-type theorems to construct $X$ in terms of the so-called height process associated to $\xi$. For convenience and to introduce more notation, we give a brief overview here. 

Denote by $(\hat{\xi}^{(t)},0\leq r\leq t)$ the time reversed process at time $t$, that is $\hat{\xi}_r^{(t)}:=\xi_t-\xi_{(t-r)-}$, and let $\hat{S}_r^{(t)}:=\sup_{s\leq r} \hat{\xi}_s^{(t)}$.
We define $H_t$ as the local time at level 0, at time $t$ of the process $\hat{S}^{(t)}-\hat{\xi}^{(t)}$. Because the reversed process has a different point from which is reversed at each time, the process $H$ does not behave in a monotone way.
The process $(H_t,t\geq 0)$ is called the $\psi$-height process, which, under assumption (\ref{grey}), is continuous.
There exists a notion of local time up to time $t$ of $H$ at level $a\geq 0$, henceforth denoted by $L_t^a$. Specifically, the family $(L^a_t, a,t\geq 0)$ satisfies 
\[
\int_0^t g(H_s){\rm d}s = \int_0^\infty g(a)L^a_t{\rm d}a\qquad t\geq 0,
\]
where $g$ is a non-negative measurable function.

For $x>0$ let $T_x:=\inf\{t\geq 0, \xi_t=-x\}$.
Then the generalised Ray-Knight theorem for the $\psi$-CSBP process states that $(L^a_{T_x},a\geq 0)$ has a c\`adl\`ag modification for which 
\begin{equation*}
(L^t_{T_x},t\geq 0)\overset{d}{=} (X, \mathbb{P}_x),
\end{equation*}
that is, the two processes are equal in law.

The height process also codes the genealogy of the $\psi$-CSBP.
It can be shown that the excursions of  $H$  from 0 form  a time-homogeneous Poisson point process of excursions with respect to local time at 0. We shall use $\texttt{n}$ to denote its intensity measure. 
%
%
%
%
 If $X_0 = x$, then the total amount of local time of $H$ accumulated at zero is $x$.  Each excursion codes a real tree (see \cite{DLG}  for a precise meaning) such that the excursion that occurs after $u\leq x$ units of local time  can be thought of as the descendants of the `$u$-th' individual in the initial population.
Here we are interested in the genealogy of the conditioned process and what we will call the embedded `$T$-prolific' tree, that is the tree of the individuals that survive up to time $T$.
Conditioning the process on survival up to time $T$ corresponds to conditioning the height process to have at least one excursion above level $T$. (We have the slightly confusing, but nonetheless standard, notational anomaly that a spatial height for an excursion corresponds to the spatial height in the tree that it codes, but that this may also be seen as a time into the forward evolution of the tree.)
Let $\texttt{n}_T$ denote the conditional probability $\texttt{n}(\cdot| \sup_{s\geq 0}\epsilon_s\geq T)$ where  $\epsilon$ is a canonical   excursion of $H$ under $\texttt{n}$. 
Let $(Z_t^T,t\geq 0)$ be the process that counts the number of excursions  above level $t$ that hit level $T$ within the excursion $\epsilon$.
Duquesne and Le Gall in \cite{DLG} describe the distribution of $Z^T$ under $\texttt{n}_T$ and prove the following.
\begin{theorem}
Under {\rm $\texttt{n}_T$} the process $(Z_t^T,0\leq t {\color{black}<} T)$ is a time-inhomogeneous Markov process whose law is characterised by the following identities.
For every $\lambda>0$
{\rm
\begin{equation*}
\texttt{n}_T\left[\exp\{-\lambda Z_t^T\}\right]=1-\frac{u_t((1-{\rm e}^{-\lambda})u_{T-t}(\infty))}{u_T(\infty)},
\end{equation*}
}
and if $0\leq t< t' <T$,
{\rm 
\begin{equation*}
\texttt{n}_T\left[ \exp\{-\lambda Z_{t'}^T\}|Z_t^T\right] = \left(\texttt{n}_{T-t}\left[ \exp\{ -\lambda Z_{t'-t}^{T-t}\}\right]\right)^{Z_t^T}.
\end{equation*}
}
\end{theorem}
\noindent In essence, the second part of the above theorem shows that $Z^T$ has the branching property. However, temporal inhomogeneity means that it is a time-dependent continuous-time Galton--Watson process. 
 In \cite{DLG} it is moreover shown that, conditionally on $L_\sigma^t$ under $\texttt{n}_t$, where $\sigma$ is the length of the excursion $\epsilon$, $Z_t^T$ is Poisson distributed with intensity $u_{T-t}(\infty)L_\sigma^t$. Thinking of $L_\sigma^t$ as the mass at time $t$ in the tree of descendants of the prolific individual in the initial population that the excursion codes, we thus have a Poisson embedding of the number of prolific descendants of that one individual within the excursion.

The time-dependent continuous-time Galton--Watson process in the theorem can also be characterised as follows.
At time 0 we start with one individual.
Then the law of the first branching time, $\gamma_T$,  is given by 
\begin{equation}\label{eq:nT}
\texttt{n}_T(\gamma_T>t) = \frac{\psi(u_T(\infty))}{u_T(\infty)}\frac{u_{T-t}(\infty)}{\psi(u_{T-t}(\infty))}, \qquad t\in[0,T),
\end{equation}
and, conditionally on $\gamma_T$, the probability generating function of the offspring distribution is
\begin{equation*}
\texttt{n}_T\left[ r^{Z_{\gamma_T}^T}|\gamma_T=t\right] = 1+\frac{\psi((1-r)u_{T-t}(\infty))-(1-r)u_{T-t}(\infty)\psi'(u_{T-t}(\infty))}{u_{T-t}(\infty)\psi'(u_{T-t}(\infty))-\psi(u_{T-t}(\infty))}, \qquad r\in[0,1].
\end{equation*}
The offspring distribution when a split occurs at height $t$ in the excursion (equivalently time $t$ in the underlying genealogical tree), say $(p^{T-t}_k, k\geq 0)$, is explicitly given by the following.
We have $p^{T-t}_0=p^{T-t}_1=0$ and for $k\geq 2$
\begin{align}
p^{T-t}_k &= \frac{1}{u_{T-t}(\infty)\psi'(u_{T-t}(\infty))-\psi(u_{T-t}(\infty))}\notag\\
&\times\left\lbrace \beta u_{T-t}^2(\infty)\mathbf{1}_{\{k=2\}}+\int_0^\infty \frac{\left(  u_{T-t}(\infty)x\right)^k}{k!}{\rm e}^{-u_{T-t}(\infty) x}\Pi(dx)\right\rbrace.\label{proba}
\end{align}

Using (\ref{eq:nT}) we can compute the rate of branching at any height $t$ in the excursion.
First it is not hard to see that
\begin{equation*}
\texttt{n}_T(\gamma_T>t+s|\gamma_T>t) = \texttt{n}_{T-t}(\gamma_T>s), \qquad 0\leq t+s< T.
\end{equation*}
Hence, the rate is
\begin{equation*}
\frac{\rm d}{{\rm d}s}\texttt{n}_{T-t}(\gamma_T>s)|_{s=0} = \frac{u_{T-t}(\infty)\psi'(u_{T-t}(\infty))-\psi(u_{T-t}(\infty))}{u_{T-t}(\infty)}, \qquad t\in [0, T).
\end{equation*}

\bigskip
Again thinking of $L^t_\sigma$, $t{\color{black}<} T$ as the mass of the tree coded by the excursion, and noting that not all of this mass is prolific, we would like to characterise the non-$T$-prolific mass that has `immigrated' along the path of the prolific tree.
 We expect this to be a CSBP conditioned to die before time $T$.
 Using \eqref{semigroup}, we know that the probability that the process dies up to time $T$ is given by:
 \begin{equation*}
\mathbb{P}\left[ X_T=0| \mathcal{F}_t\right]={\rm e}^{- X_t u_{T-t}(\infty)},
 \end{equation*}
where we  assume Grey's condition \eqref{grey} to ensure that the above conditioning makes sense. 
A simple application of the Markov property tells us that the  law of $X$ conditioned to die out by time $T$ can be obtained by the following change of measure
 \begin{equation}
 \left.\frac{{\rm d}\mathbb{P}^T_x}{{\rm d}\mathbb{P}_x}\right|_{\mathcal{F}_t}=\frac{{\rm e}^{-X_t u_{T-t}(\infty)}}{{\rm e}^{-x u_T(\infty)}}, \qquad t\geq 0, x>0.
 \label{TCOM}
 \end{equation}
Indeed,  using the semigroup property of $u_t$, $t\geq 0$, it is not hard to verify that the right-hand side above is a martingale. 
We would like to understand how to characterise  the evolution of the process $(X, \mathbb{P}^T_x)$, $x>0$, a little better as the change of measure is time inhomogeneous. 

To this end, we again appeal to  It\^{o}'s formula.
Denote $v(t):=u_{T-t}(\infty)$, then for non-negative, twice differentiable and compactly supported functions $f$, after a routine, albeit lengthy, application of It\^o's formula we get, for $t\geq 0$ and $x>0$,
\begin{equation*}
\begin{split}
{\rm e}^{-X_t v(t)}f(X_t)=& {\rm e}^{-x v(0)}f(x)+M_t \\
&+\int_0^t X_s {\rm e}^{-X_s v(s)}\Big( -f(X_s)v'(s)-\alpha v(s)f(X_s)+\alpha f'(X_s)\\
&\hspace{4cm}+\beta v^2(s)f(X_s)-2\beta v(s)f'(X_s)+\beta f''(X_s) \Big)\d t\\
&+\sum_{s\leq t}\Big( {\rm e}^{-X_s v(s)}f(X_s)-{\rm e}^{-X_{s-} v(s-)}f(X_{s-})\\
&\hspace{2cm}+v(s-) {\rm e}^{-X_{s-}v(s-)}f(X_{s-})\Delta X_s-{\rm e}^{-X_{s-}v(s-)}f'(X_{s-})\Delta X_s\Big),
\end{split}
\end{equation*}
where $M_t$, $t\geq 0$, represents  the martingale terms.
Taking expectations we get
\begin{equation*}
\begin{split}
&\mathbb{E}_x\left[{\rm e}^{-X_t v(t)}f(X_t)\right]\\
&= {\rm e}^{-x v(0)}f(x)\\&+
\mathbb{E}_x\Bigg[\int_0^t \Big\{X_s {\rm e}^{-X_s v(s)}\Big( -f(X_s)v'(s)-\alpha v(s)f(X_s)+\alpha f'(X_s)\\
&\hspace{5cm}+\beta v^2(s)f(X_s)-2\beta v(s)f'(X_s)+\beta f''(X_s) \Big)\\
&\hspace{2cm}+\int_0^\infty X_{s}\Big( {\rm e}^{-(X_{s}+y) v(s)}f(X_{s}+y)-{\rm e}^{-X_{s} v(s)}f(X_{s})\\
&\hspace{5cm}+v(s) {\rm e}^{-X_{s}v(s)}f(X_{s})y-{\rm e}^{-X_{s}v(s)}f'(X_{s})y\Big)\Pi(\d y)\Big\} {\rm d}s\Bigg].\\
\end{split}
\end{equation*}
Gathering terms, making use of the expression for $\psi$ in \eqref{mechanism} and that 
\[-\frac{\partial}{\partial s}u_{T-s}(\infty)+\psi(u_{T-s}(\infty))=0
\]
we have, for $t\geq 0$ and $x>0$, the Dynkin formula
\begin{equation*}
\begin{split}
\mathbb{E}^T_x\left[ f(X_t)\right]&=\mathbb{E}_x\left[ \frac{{\rm e}^{-u_{T-t}(\infty) X_t}}{{\rm e}^{-u_T(\infty) x}}f(X_t)\right] \\
&= f(x)+\mathbb{E}_x\left[ \int_0^t \frac{{\rm e}^{-u_{T-s}(\infty) X_s }}{{\rm e}^{-u_T(\infty) x}}\left(\hat{\mathcal{L}}_{T-s} f(X_s)\right) dt\right]\\
&=f(x)+\mathbb{E}^T_x\left[ \int_0^t\hat{\mathcal{L}}_{T-s} f(X_s) \d t\right],
\end{split}
\end{equation*}
where the infinitesimal generator is given by 
\begin{align}
\hat{\mathcal{L}}_{T-t} f(x)&= \psi'(u_{T-t}(\infty))xf'(x)+ \beta x f''(x)+ \notag\\
&\hspace{3cm} x\int_0^\infty(f(x+y)-f(x)-yf'(x)){\rm e}^{-u_{T-t}(\infty) y}\Pi(\d y).
\label{Tdie}
\end{align}
For comparison, consider
the generator of a CSBP with Esscher transformed branching mechanism $\psi_\lambda$, which   is given by 
\begin{equation}
\mathcal{L}_\lambda f(x)= \psi'(\lambda)xf'(x) + \beta xf''(x) + x \int_0^\infty(f(x+y)-f(x)-yf'(x)){\rm e}^{-\lambda y}\Pi(\d y)
\label{L}
\end{equation}
for suitably smooth and integrable functions $f$. Recall that the CSBP  with generator \eqref{L} is subcritical providing $\psi'(\lambda)>0$ and, taking account of \eqref{meangrowth}, the greater this value, the `more subcritical' it becomes. It appears that $\hat{\mathcal{L}}_{T-t}$ has the form of an Esscher transformed branching mechanism based on $\psi$ where the parameter shift is controlled by $u_{T-t}(\infty)$, which explodes as $t\to T$. 
Said another way, if we define $V^T_t(\theta)$, $0\leq t<T$, $x,\theta\geq 0$ as the exponent satisfying
\[
\mathbb{E}^T_x[{\rm e}^{-\theta X_t}] = {\rm e}^{-x V^T_t(\theta)}, 
\]
then 
\begin{equation}
V^T_t(\theta) = u_t(\theta + u_{T-t}(\infty)) - u_{T}(\infty) .
\label{V^T_t}
\end{equation}

Recalling that we are assuming Grey's condition \eqref{grey} for a (sub)critical process, we  note  from \eqref{semigroup} that 
\[
\lim_{T\to\infty}u_T(\infty) = 0.
\]
 In that case,  the density in  \eqref{TCOM} tends to unity as $T\to\infty$.

\bigskip

We conclude this section by returning to Theorem \ref{T}. The discussion in this  section shows that {\color{black} in the (sub)critical case}, the components of the SDE \eqref{TcoupledSDE} mimic precisely the description of the $T$-skeleton in the previous section.  In particular, the first three integrals in \eqref{TcoupledSDE} indicate that once mass is created in the SDE, it evolves as a time-dependent CSBP with generator $\hat{\mathcal{L}}_{T-t}$. Moreover, the evolution of the skeleton $Z^T$ as described in \eqref{TcoupledSDE}, matches precisely the dynamics of the $T$-prolific skeleton described in the previous section (for which we pre-emptively used the same notation), which is a time-dependent continuous-time Galton--Watson process. Indeed, the branching rate and the time-dependent offspring distribution of both match.

\begin{rem}\rm It is important to note that  even though the time-dependent $T$-prolific skeleton is inspired by the height process, the description does not require $\psi$ to be a (sub)critical branching mechanism. Indeed, only requiring Grey's condition to be satisfied ensures that the branching rate and offspring distribution of this section are well defined. Similarly, we can apply the change of measure in \eqref{TCOM} for any CSBP satisfying \eqref{grey}, and get a time-dependent CSBP with generator $\hat{\mathcal{L}}_{T-t}$. 
Although the results of Theorem \ref{T} were motived by Duquesne and Le Gall \cite{DLG}, we can see that the theorem can be stated in a more general setting, and thus extends the existing family of finite-horizon decompositions for CSBPs. 
\end{rem}

\section{Finite-time horizon Skeleton: Proof of Theorem \ref{T}}\label{proof2}

Now that we understand that the mathematical structure of \eqref{TcoupledSDE} is little more than a time-dependent version of \eqref{coupledSDE},   the reader will not be surprised by the claim that the proof of strong uniqueness to  \eqref{TcoupledSDE}  as well as parts  (i) and (ii) of Theorem \ref{coupled} pass through almost verbatim, albeit needing some minor adjustments for additional time derivatives of $u_{T-t}(\infty)$ in e.g. \eqref{Fint} and \eqref{Ypde}, which plays the role of $\lambda$. To avoid repetition, we simply leave the proof of these two parts as an exercise for the reader.  

\medskip

On the event $\{Z_0 =0\}$, which is concurrent with the event $\{Z_t = 0, \, 0\leq t<T\}$ close inspection of \eqref{TcoupledSDE} allows us to  note that $\Lambda$ is generated by an SDE with time-varying coefficients. Indeed, standard arguments show that conditional on $\{Z_0 =0\}$, $\Lambda$ is a time inhomogeneous Markov processes.  

Suppose that we write $\mathbf{P}^T_{x,n}$, $x\geq 0, n\in\mathbb{N}_0$ for the law of the Markov probabilities corresponding to the solution of \eqref{TcoupledSDE}. Moreover, we will again abuse this notation in the spirit of \eqref{abuse} and write $\mathbf{P}^T_x$, $x\geq 0$, when $Z^T_0$ is randomised to be an independent Poisson random variable with rate $u_T(\infty)x$.  We can use  part (i) and (ii) of Theorem \ref{T}, together with  \eqref{TCOM} to deduce that
\begin{align}
  \mathbf{E}^T_{x}\left[ {\rm e}^{-\eta \Lambda_t}|Z_0=0\right]& = \frac{ \mathbf{E}^T_{x}\left[ {\rm e}^{-\eta \Lambda_t}, 
  \, Z_0=0\right]}{ \mathbf{P}^T_{x}(Z_0 = 0)}\notag\\
  & = \frac{ \mathbf{E}^T_{x}\left[ {\rm e}^{-\eta \Lambda_t}, 
  \, Z_t=0\right]}{ \mathbf{P}^T_{x}(Z_0 = 0)}\notag\\
  & = {\rm e}^{u_T(\infty)x} \mathbf{E}^T_{x}\left[ {\rm e}^{-(\eta +u_{T-t}(\infty)) \Lambda_t}\right]\notag\\
  & = \mathbb{E}^T_{x}\left[ {\rm e}^{-\eta X_t}\right].\label{Z=0}
  \end{align}
  This tells us that the semigroups of $\Lambda$ conditional on $\{Z_0 =0\}$ and $X$ conditional to become extinct by time $T$ agree. Part (iii) of Theorem \ref{T} is thus proved.
\hfill$\square$

\section{Thinning the skeleton to a spine: Proof of Theorem \ref{skeletontospine}}\label{proof3}

The aim of this section is to recover the unique solution to \eqref{surv} as a weak limit of \eqref{TcoupledSDE} in the sense of Skorokhod convergence. To this end, we assume throughout the conditions of Theorem \ref{skeletontospine}, in particular that $\psi$ is a critical or subcritical branching mechanism and Grey's condition \eqref{grey} holds.

There are three main reasons why we should expect this result and these three reasons pertain to the three structural features of the skeleton decomposition: The feature of Poisson embedding, the Galton--Watson skeleton and the branching immigration from the skeleton with an Esscher transformed branching mechanism. Let us dwell briefly on these heuristics.

First,  let us consider the behaviour of the skeleton $(Z_t^T, t<T)$ as $T\to\infty$. As we are assuming that $\psi$ is a (sub)critical branching mechanism, it holds that $\lim_{T\rightarrow\infty}u_T(\infty) = 0$ as $T\rightarrow\infty$.
Thus, recalling that $Z^T_0\sim \texttt{Po}(u_T(\infty)x)$, i.e. independent and Poisson distributed with parameter $u_T(\infty)x$,  and hence conditioning on survival to time $T$ in the skeletal decomposition is tantamount to conditioning on the event $\{Z^T_0\geq 1\}$, we see that 
 \begin{equation}\label{condprob}
 \varpi^{x,T}_k: = \mathbf{P}^T_x[Z_0=k|Z_0\geq 1] = \frac{(u_{T}(\infty)x)^k}{k!}\frac{{\rm e}^{-u_T(\infty) x}}{1-{\rm e}^{-u_{T}(\infty)x}}, \qquad k\geq 1.
 \end{equation}
We thus see that the probabilities \eqref{condprob} all tend to zero unless $k=1$ in which case  the limit is unity.
Moreover, Theorem \ref{T} (ii) and (iii) imply that the law of  $(\Lambda^T_t, 0\leq t<T)$ conditional on $(\mathcal{F}^{\Lambda_t^T}\cap\{Z^T_0\geq 1\}, 0\leq t<T)$ corresponds to the law of the $\psi$-CSBP, $X$, conditioned to survive until   time $T$. Intuitively, then, one is compelled to believe that, in law, there is asymptotically a single skeletal contribution to the law of $X$ conditioned to survive.

 Second,  considering (\ref{eq:nT}), it follows from l'Hospital's rule that the rate at which aforementioned most common recent ancestor  branches  begins to slow down since
 \begin{equation*}
 \frac{\psi(u_T(\infty))}{u_T(\infty)}\frac{u_{T-t}(\infty)}{\psi(u_{T-t}(\infty))}\rightarrow 1,
 \end{equation*}
 as $T\rightarrow \infty$.
 Furthermore, considering \eqref{proba}, we also get that in the limit of the offspring distribution $\lim_{T\to\infty}p^T_k=0$ for all $k\geq 0$.  What we are thus observing is a thinning, in the weak sense, of the skeleton, both in terms of the number of branching events as well as the number of offspring.

Thirdly, we consider  the mass that immigrates from the skeleton. For a fixed $T$, it evolves as a $\psi$-CSBP conditioned to die before time $T$. We recall that conditioning to die before time $T$ is tantamount to the change of measure given in \eqref{TCOM}. It is easy to see that  as $T\to\infty$, the density in this change of measure converges to unity and hence immigrating mass, in the weak limit, should have the evolution of a $\psi$-CSBP. 

\medskip

With all this evidence in hand,  Theorem \ref{skeletontospine} should now take on a natural meaning.  We give its proof below.

\begin{proof}[Proof of Theorem \ref{skeletontospine}]
According to Theorem 2.5 on p.167 of \cite[Chapter 4]{EthKur}, if $E$ is a locally compact and separable metric space, 
   $\mathcal{P}^T:  =(\mathcal{P}_t^T, t\geq 0)$, $T>0$, is a sequence of Feller semigroups on $C_0(E)$ (the space of continuous functions on $E$ vanishing at $\infty$, endowed with the supremum norm),   $\mathcal{P}: = (\mathcal{P}_t, t\geq 0)$ is a Feller semigroup on  $C_0(E)$ such that, for $f\in C_0(E)$, with respect to the supremum norm on the space $C_0(E)$,
  \begin{equation}
  \lim_{T\to\infty} \mathcal{P}_t^T f=\mathcal{P}_t f, \quad t\geq 0,
\label{Fellercgce}
  \end{equation}
  and moreover,  $(\nu^T,T>0)$ is  a sequence of Borel probability measures on  $E$
 such that $\lim_{T\to\infty}\nu^T =\nu$ weakly  for some probability measure $\nu$  then, with respect to the Skorokhod topology on $\mathbb{D}([0,\infty), E)$, $\Xi^T$ converges weakly to $\Xi$, where $(\Xi^T,T>0)$  are  the strong Markov processes associated to   $(\mathcal{P}^T, T>0)$ with initial law $(\nu^T,T>0)$ and $\Xi$ is the strong Markov processes associated to   $\mathcal{P}$ with initial law $\nu$, respectively.

Note that such weak convergence results would normally require a tightness criterion, however, having the luxury of \eqref{Fellercgce}, where $\mathcal{P}$ is a Feller semigroup, removes this condition and this will be the setting in which we are able to apply the conclusion of the previous paragraph.

\medskip

Fix $t_0>0$. We want to prove the weak convergence result in the finite time window $[0,t_0]$.
In order to introduce the role of $\mathcal{P}^T$, $T>0$, in our setting, we will abuse yet further previous notation and 
 define $\mathbf{P}^T_{x,n, s}$, $x\geq 0$, $n\geq 0$, $s\in[0,t_0]$  to be the Markov probabilities associated to the three-dimensional process $(\Lambda_t, Z_t, \tau_t)$, $0\leq t\leq t_0$, whenever $t_0<T$, where $(\Lambda_t, Z_t)$, $0\leq t\leq t_0$ is the weak solution to \eqref{TcoupledSDE}, and $\tau_t: = t$, $0\leq t\leq t_0$.
 Consistently with previous related notation, we have $\mathbf{P}^T_{x,n, 0} = \mathbf{P}^T_{x,n}$, $x\geq 0$, $n\geq 0$, 
 Now 
define the associated time-dependent 
semigroup for the three-dimensional process, for $t\geq 0$ and $f\in C_0([0,\infty)\times\mathbb{N}_0\times[0,\infty))$, such that $\texttt{P}^T_t[f](x,n,s)=f(x,n,s)$ when $T\leq t_0$, and when $T> t_0$ we have
\begin{align*}
\texttt{P}^T_t[f](x,n,s)&: = \mathbf{E}^T[f(\Lambda_{(t\vee s)\wedge t_0}, Z_{(t\vee s)\wedge t_0}, \tau_{(t\vee s)\wedge t_0})| \Lambda_{s} =x, Z_{s} = n,\tau_{s}=s]
\end{align*}
for $(x,n,s)\in[0,\infty)\times\mathbb{N}_0\times[0,t_0]$, and $\texttt{P}^T_t[f](x,n,s): = f(x,n,s)$ for $(x,n,s)\in[0,\infty)\times\mathbb{N}_0\times(t_0,\infty)$.
We take  $\mathcal{P}^T=\texttt{P}^T$.
In order to verify the Feller property of $\texttt{P}^T_t[f](x,n,s)$ we need to check two things (cf. Proposition 2.4 in Chapter III of \cite{RY}): 
\begin{itemize}
\item[(i)] For each $t\geq 0$, the function $(x,n,s)\mapsto \texttt{P}^T_t[f](x,n,s)$ belongs to $C_0([0,\infty)\times\mathbb{N}_0\times[0,\infty))$, for any $f$ in that space.
\item[(ii)] For all $f\in C_0([0,\infty)\times\mathbb{N}_0\times[0,\infty))$ and for each $(x,n,s)\in [0,\infty)\times\mathbb{N}_0\times[0,\infty)$, we have $\lim_{t\downarrow s}\texttt{P}^T_t[f](x,n,s) = f(x,n,s)$. 
\end{itemize}
Note that when $T\leq t_0$, or $s\geq t_0$, or $t\leq s\leq t_0$, we have $\texttt{P}^T_t[f](x,n,s)=f(x,n,s)$.
Since $f\in C_0([0,\infty)\times\mathbb{N}_0\times[0,\infty))$, both (i) and (ii) are trivially satisfied.
We can also notice that the case when $T>t_0$, $s\leq t_0$ and $t\geq t_0$ reduces to the case of $t=t_0$,
hence in order to show the Feller property of $\texttt{P}^T_t[f](x,n,s)$ we can restrict ourselves to the case of $s\leq t\leq t_0<T$.

By denseness    of  the sub-algebra generated by exponential functions (according to the uniform topology) in $C_0(E)$, it suffices to check, for (i), that 
\begin{equation}
(x,n,s)\mapsto  \mathbf{E}_{x,n,s}^T\left[ {\rm e}^{-\gamma \Lambda_{t}-\theta Z_{t} - \varphi\tau_{t}}\right], \quad s\leq t\leq t_0<T,
\label{check1}
\end{equation}
belongs to $C_0([0,\infty)\times\mathbb{N}_0\times[0,\infty))$ and, for (ii), that 
\begin{equation}
\lim_{t\downarrow 0} \mathbf{E}_{x,n,s}^T\left[ {\rm e}^{-\gamma \Lambda_{t}-\theta Z_{t} - \varphi\tau_{t}}\right] = {\rm e}^{-\gamma x-\theta n - \varphi s}, \quad s\leq t\leq t_0<T.
\label{check2}
\end{equation}
To this end, note that 
\begin{equation}
 \mathbf{E}_{x,n,s}^T\left[ {\rm e}^{-\gamma \Lambda_t-\theta Z_t - \varphi\tau_t}\right]  =
  \mathbf{E}_{x,n}^{T-s}\left[ {\rm e}^{-\gamma \Lambda_{t-s}-\theta Z_{t-s} }\right]
  {\rm e}^{-\varphi t} , \qquad s\leq t\leq t_0<T.
  \label{I1}
  \end{equation}
In order to evaluate expectation on the right-hand side above,  we  want to work with an appropriate representation of the  unique weak solution to \eqref{TcoupledSDE}.  We shall do so by following the example of how the weak solution to \eqref{coupledSDE} was identified in the form \eqref{weaksolutionrepresentation}. 

As before, we need to introduce additionally marked versions of the Poisson random measures ${\texttt N}_T^{1}$ and ${\texttt N}_T^{2}$, as well as an additional Poisson random measure ${\texttt N}_T^*$. We will insist that Poisson random measure $ {\texttt N}^1_T(\d s, \d r, \d{j}, \d\omega)$  on  $[0,T)\times[0,\infty) \times\mathbb{N}_0\times\mathbb{D}([0,\infty), \mathbb{R})$  has  intensity 
 ${\rm d} s\otimes r  {\rm e}^{-u_{T-s}(\infty) r} \Pi({\rm d}r) \otimes\sharp({\rm d}j)\otimes \mathbb{P}^{T-s}_r(\d \omega)$, Poisson random measure $ {\texttt N}^{2}_T({\rm d}s,{\rm d}r,{\rm d}k,{\rm d}j, \d\omega)$   on  $[0,T)\times[0, \infty)\times \NN_0\times\mathbb{N}\times \mathbb{D}([0,\infty), \mathbb{R})$ has intensity 
 \begin{align*}
&  
q^{T-s}
{\rm d}s\otimes   \eta_k^{T-s}({\rm d}r)\otimes p^{T-s}_k\sharp({\rm d}k)\otimes \sharp({\rm d}j)\otimes \mathbb{P}^{T-s}_r(\d \omega)
 \end{align*}
and Poisson random measure 
${\texttt N}^*_T(\d s, \d{j}, \d \omega)$ has intensity $2\beta\d s\otimes \sharp(\d j)\otimes\mathbb{Q}^{T-s}(\d\omega)$
on $[0,T)\times\mathbb{N}_0\times\mathbb{D}([0,\infty), \mathbb{R})$, where $\mathbb{Q}^T$ is the excursion measure associated to $\mathbb{P}^T_r$, $r\geq 0$, satisfying 
\begin{equation}
\mathbb{Q}^T(1 - {\rm e}^{-\gamma \omega_t}) = V^T_t(\gamma),\qquad \gamma >0,
\end{equation} 
for $0\leq t<T$, where $V^T_t$ was defined in \eqref{V^T_t}.
To recall some of the notation used in these rates, see \eqref{timeeta} and \eqref{rate}.

If the pair $(\Lambda, Z)$ has law $\mathbf{P}^T_{x,n}$, then we can write 
\begin{equation}
\Lambda_t =X_t + D_t,\qquad t< T,
\label{X+D}
\end{equation}
where $X$ is autonomously independent with law $\mathbb{P}^T_x$ and, given ${\texttt N}^{2}$,  $D$ is the uniquely identified (up to almost sure modification) `dressed skeleton' described by 
\begin{align*}
D_t&=  \int_{0}^{t}\int_{0}^{\infty}\int_{1}^{Z_{s -}}\int_{\mathbb{D}([0,\infty), \mathbb{R})}  
\omega_{t-s} \, {\texttt N}_T^1(\d s,\d r,\d{j}, \d\omega)\\
&\hspace{2cm}
      + \int_{0}^{t}\int_{0}^{\infty}\int_{0}^{\infty}\int_{1}^{Z_{s- }} \int_{\mathbb{D}([0,\infty), \mathbb{R})} \omega_{t-s} \, {\texttt N}_T^{2}(\d s,\d r,\d{k},\d{j}, \d\omega)\\
      &
   \hspace{4cm}+  \int_0^t \int_{1}^{Z_{s- }}\int_{\mathbb{D}([0,\infty), \mathbb{R})}  \omega_{t-s} \, {\texttt N}_T^*(\d s, \d{j}, \d \omega),
\end{align*}
where $Z_0 = n$.
The verification of this claim follows almost verbatim the same as for \eqref{coupledSDE} albeit with obvious change to take account of the time-varying rates. We therefore omit the proof and leave it as an exercise for the reader.

With the representation \eqref{X+D}, as $Z$ is piecewise constant, we can condition on the sigma-algebra generated by ${\texttt N}_T^{2}$ and show, using Campbell's formula in between the jumps of $Z$, that, for $0\leq t<T$, $\gamma, \theta\geq 0$, $x\geq 0$ and $n\in\mathbb{N}_0$,
  \begin{align}
 &\mathbf{E}_{x,n}^{T}\left[ {\rm e}^{-\gamma \Lambda_{t}-\theta Z_{t} }\right]\notag\\
 &
= {\rm e}^{-x V^{T}_t(\gamma)}\mathbf{E}^T_{0,n}\left[ {\rm e}^{-\theta Z_t- \int_0^t Z_{v}
\phi_{u_{T-v}(\infty)}(V^{T-v}_{t-v}(\gamma)){\rm d}v 
}
\prod_{w\leq t}\left(\int_0^\infty {\rm e}^{- rV^{T-w}_{t-w}(\gamma)} \eta^{T-w}_{\Delta Z_w+1} (\d r)\right)
\right] 
\label{jumpproduct}
\end{align} 
where 
\[
V^T_t(\gamma): = u_{t}(\gamma + u_{T-t}(\infty))-u_{T}(\infty), \qquad 0\leq t<T,
\]
and, for $\lambda, z \geq 0$,
\[
\phi_\lambda(z) = 2\beta z + \int_0^\infty (1- {\rm e}^{-z r})r{\rm e}^{-\lambda r}\Pi(\d r).
\]

Given the identities \eqref{I1} and \eqref{jumpproduct}, the two required verifications in \eqref{check1} and \eqref{check2}  follow easily as direct consequence of continuity and bounded convergence in \eqref{jumpproduct}. 


The target semigroup $\mathcal{P}$   on $f\in C_0([0,\infty)\times\mathbb{N}_0\times[0,\infty))$ is defined  as follows. For fixed $n\in\mathbb{N}_0$, $x\geq 0$, let $\mathbb{P}^{(n)}_{x}$  be the law of the homogeneous Markov process described by the weak  solution to
\begin{align*}
 X_t = x&+\alpha\int_0^t X_{s-}\d s+\sqrt{2\beta}\int_0^t\int_0^{X_{s-}} W(\d s,\d u)
+ \int_0^t\int_0^\infty\int_0^{X_{s-}} r \tilde{N}(\d s,\d r,\d u) \notag\\
&+\, \int_0^t\int_0^\infty r N^{(*,n)}(\d s,\d r) + 2n \beta t,\qquad t\geq0,
 \end{align*} 
 with $W,N$ and $N^{(*,n)}$ is a Poisson random measure  
 on $[0,\infty)^2\times\mathbb{D}([0,\infty), \mathbb{R})$  with intensity measure $n\d s\otimes r\Pi(\d r)\otimes\mathbb{P}_r(\d\omega)$.  Note, we have at no detriment to consistency that $\mathbb{P}^{(0)}_{x}$ can be replaced by $\mathbb{P}_{x}$. Then, we take the role of  $\mathcal{P}_t$ played by the semigroup $\texttt{P}_t^\uparrow$  given 
 by 
\begin{align*}
\texttt{P}^\uparrow_t[f](x,n,s)&: = \mathbf{E}^{(n)}[f(X_t, n, \tau_t)| X_s =x, \tau_s =s],
\end{align*}
for $0\leq s\leq t\leq t_0$, $n\geq 0$, and $\texttt{P}^\uparrow_t[f](x,n,s): =f(x,n,s)$ otherwise.
Here $f\in C_0([0,\infty)\times\mathbb{N}_0\times[0,\infty))$,
and  $\tau_t=t$,  as above. 
Notice $(X, \mathbf{P}^{(n)}_{x})$   is a branching process with immigration, whose Laplace transform is given by 
\[
 \mathbf{E}^{(n)}_{x}({\rm e}^{-\gamma X_t})= {\rm e}^{-xu_{t}(\gamma)- n\int_0^{t} \phi_0(u_{t-v}(\gamma))\d v},
 \qquad \gamma\geq 0.
 \]
  From this, it is easily seen that  $\texttt{P}^\uparrow_t$ 
is Feller as well.

Lastly, for each $T\geq 0$ we take $\nu^T$  the measure on $[0,\infty)\times\mathbb{N}_0\times[0,\infty)$ given for each $x\geq 0$  by 
$\delta_x\otimes \pi^{T,x} \otimes \delta_0$, with
$
\pi^{T,x}(\cdot) = \sum_{n\geq 1}\varpi_n^{T,x} \delta_n(\cdot). 
$
Recall from \eqref{condprob} that  $\pi^{T,x}$ converges weakly, as $T\to\infty$, to the measure $\delta_1(\cdot)$ on $\mathbb{N}_0$, hence $\nu^T$ converges weakly  to $\nu:=\delta_x\otimes \delta_1 \otimes \delta_0$. 
Thus, in order to invoke Theorem 2.5  in \cite[Chapter 4]{EthKur},  we just need to check the analogue of \eqref{Fellercgce} in our setting. 

To this end, notice first that we can restrict ourselves to $0\leq s\leq t\leq t_0$, since otherwise $\texttt{P}^\uparrow_t[f](x,n,s)=\texttt{P}^T_t[f](x,n,s)$ by definition.
Then note from \eqref{rate} that $q^T\to0$ as $T\to\infty$, and this yields that, under $\mathbf{P}^T_{0,n}$, process $Z$  converge in probability  uniformly on  $[0,t]$,   as $T\to\infty$  (cf. Theorem 6.1, Chapter 1, p28 of \cite{EthKur}) to the constant process $Z_s\equiv n$, $s\leq t$. 
Referring back to \eqref{jumpproduct},   the  continuity in $T$ of the deterministic quantities  as they appear on the right-hand side and the previously mentioned uniform  convergence of $(Z, \mathbf{P}^T_{0,n})$ together  imply that, for $x\geq 0,0\leq s\leq t\leq t_0$, $n\in\mathbb{N}_0$,
\begin{align*}
\lim_{T\to\infty}\texttt{P}_t^T[f_{\gamma,\theta,\varphi}](x,n, s)& = {\rm e}^{-\varphi t}
\lim_{T\to\infty}\mathbf{E}^{T-s}_{x,n}\left[ {\rm e}^{-\gamma \Lambda_{t-s}-\theta Z_{t-s} }\right] \\
& = \texttt{P}^\uparrow_t[f_{\gamma,\theta,\phi}](x,n,s)\\
&= {\rm e}^{-xu_{t-s}(\gamma)-\theta n - n\int_0^{t-s} \phi_0(u_{t-s-v}(\gamma))\d v-\varphi t},
\end{align*}
where 
\[
f_{\gamma,\theta, \varphi}(x,n,s)
:= {\rm e}^{-\gamma x- \theta n - \varphi s},\qquad \gamma,\theta, \varphi, x,s\geq 0, n\in\mathbb{N}_0.
\]

To conclude, it is thus enough to prove that  this convergence holds uniformly in $x\geq 0,0\leq s\leq t$, $n\in\mathbb{N}_0$, where $t\leq t_0$.  
Consider fixed $R>0$ and $N\in \mathbb{N}$. Since $V_t^T(\gamma)$ defined above is nonnegative and,  for each $n\in \mathbb{N}_0$,  $Z_t\geq Z_0=n$, $t\geq0$, a.s. under $\mathbf{P}^T_{0,n}$, for all $T>0$, using the triangle inequality, we have
\begin{equation*}
\sup_{x\geq 0,s\leq t,n\in \mathbb{N}_0}|\texttt{P}^T_t[f_{\gamma,\theta,\varphi}](x,n,s)-\texttt{P}^\uparrow_t[f_{\gamma,\theta,\varphi}](x,n,s)|\leq A_R(T) + A^R(T) + B_N(T)+B^N
\end{equation*}
where we have set:
$$A_R(T):= \sup_{x\leq R,s\leq t}| {\rm e}^{-x V^{T-s}_{t-s}(\gamma)}-   {\rm e}^{-xu_{t-s}(\gamma) } | \, , \quad  A^R(T):=   \sup_{s\leq t} {\rm e}^{-R V^{T-s}_{t-s}(\gamma)}+ \sup_{s\leq t}   {\rm e}^{-R u_{t-s}(\gamma) }   ,$$
\begin{equation*}
\begin{split}
B_N(T):= & \sup_{n\leq N,s\leq t}  \bigg| \mathbf{E}^{T-s}_{0,n}\Bigg[ {\rm e}^{-\theta Z_{t-s}- \int_0^{t-s} Z_{v}
\phi_{u_{T-s-v}(\infty)}(V^{T-s-v}_{t-s-v}(\gamma)){\rm d}v }\\
&\hspace{2.5cm}\prod_{w\leq t-s}\left(\int_0^\infty {\rm e}^{- rV^{T-s-w}_{t-s-w}(\gamma)} \eta^{T-s-w}_{\Delta Z_w+1} (\d r)\right)
\Bigg]
 -   {\rm e}^{-\theta n - n\int_0^{t-s} \phi_0(u_{t-s-v}(\gamma))\d v}
 \bigg|\\
 \end{split}
 \end{equation*}
and  $B^N =2  {\rm e}^{-\theta N}.$ 
Firstly, it is not hard to see that 
\begin{equation*}
 A_R(T)\leq \sup_{s\leq t} R|  V^{T-s}_{t-s}(\gamma)- u_{t-s}(\gamma)  |=  R\sup_{s\leq t}   |u_{t-s}(\gamma+  u_{T-t}(\infty))-  u_{T-s}(\gamma)-u_{t-s}(\gamma)|.
 \end{equation*}
 The identity ${\partial u_s(\theta)}/{\partial \theta}= {\rm e}^{-\int_0^s\psi'(u_r(\theta))\d r}$ (see (12.12) in  \cite[Chapter 12]{Kbook}) and the fact that $\psi'(\theta)\geq 0$ allows us to 
 estimate $|u_{t-s}(\gamma+  u_{T-t}(\infty))- u_{t-s}(\gamma)|$ by $u_{T-t}(\infty)$.
 Recalling that $u_T(\infty)\rightarrow 0$ and $u_T(\gamma)\rightarrow 0$ as $T\rightarrow\infty$, it follows that $A_R(T)$ tends to $0$ as $T\to \infty$,  for each $R>0$.
Next, since $(s,\gamma)\mapsto u_s(\gamma)$ is increasing in $\gamma$ and decreasing in $s$, we have
 $$ V^{T-s}_{t-s}(\gamma) \geq u_{t-s}(\gamma +  u_{T-t}(\infty))- u_{T-t}(\infty)\geq \inf_{\lambda\leq  u_{T-t}(\infty)}u_{t}(\gamma+\lambda)-\lambda,$$
 which, for  $T$ sufficiently large,  is bounded from below by $u_t(\gamma)/2>0$.  Fix $\varepsilon>0$. Choosing $R>0$ such that 
 ${\rm e}^{-R u_t(\gamma)/2}+   {\rm e}^{-R u_{t}(\gamma) } \leq \varepsilon$ we thus get
  $$\limsup_{T\to \infty} A^R(T)\leq \varepsilon.$$ 
With regard to the term $B_N(T)$, we have
  \begin{align}
B_N(T)\leq &\,  \max_{n\leq N}\sup_{s\leq t} \mathbf{E}^{T-s}_{0,n} \bigg[ 1\wedge \left(  \sup_{v\leq t} Z_v   \int_0^{t-s} |
\phi_{u_{T-s-v}(\infty)}(V^{T-s-v}_{t-s-v}(\gamma)) -  \phi_0(u_{t-s-v}(\gamma)) | {\rm d}v
  \right) \bigg]\notag \\
& \quad + \max_{n\leq N} \sup_{s\leq t} \mathbf{E}^{T-s}_{0,n} \bigg[   \bigg|  {\rm e}^{-\theta Z_{t-s}}
- {\rm e}^{-\theta n }
 \bigg|   \bigg]\notag \\
  & \quad +  \max_{n\leq N}  \sup_{s\leq t}  \mathbf{E}^{T-s}_{0,n} \bigg[ \bigg| 1-\prod_{w\leq s}\left(\int_0^\infty {\rm e}^{- rV^{T-s-w}_{t-s-w}(\gamma)} \eta^{T-s-w}_{\Delta Z_w+1} (\d r)\right) \bigg|   \bigg]\notag\\
  \leq &\,  \max_{n\leq N}\sup_{s\leq t} \mathbf{E}^{T-s}_{0,n} \bigg[ \sup_{s'\leq t}1\wedge \left(  \sup_{v\leq t} Z_v   \int_0^{t-s'} |
\phi_{u_{T-s'-v}(\infty)}(V^{T-s'-v}_{t-s-v}(\gamma)) -  \phi_0(u_{t-s'-v}(\gamma)) | {\rm d}v
  \right) \bigg]\notag \\
& \quad + \max_{n\leq N} \sup_{s\leq t} \mathbf{E}^{T-s}_{0,n} \bigg[  \sup_{s'\leq t} \bigg|  {\rm e}^{-\theta Z_{t-s'}}
- {\rm e}^{-\theta n }
 \bigg|   \bigg]\notag \\
  & \quad +  \max_{n\leq N}  \sup_{s\leq t}  \mathbf{E}^{T-s}_{0,n} \bigg[ \sup_{s'\leq t}\bigg| 1-\prod_{w\leq s'}\left(\int_0^\infty {\rm e}^{- rV^{T-s'-w}_{t-s'-w}(\gamma)} \eta^{T-s'-w}_{\Delta Z_w+1} (\d r)\right) \bigg|   \bigg]. 
  \label{3 terms}
  \end{align}  
The first term on the right-hand side above is bounded by  $$ \max_{n\leq N} \sup_{s\leq t}  \mathbf{P}^{T-s}_{0,n}  ( \sup_{v\leq t} Z_v>n)+   1\wedge  \left(N t \sup_{w\leq t}  |
\phi_{u_{T-w}(\infty)}(V^{T-w}_{t-w}(\gamma)) -  \phi_0(u_{t-w}(\gamma)) | 
 \right) $$ and hence  goes to $0$ for each $N$ as $T\to \infty$.  On the other hand, as a function of $(Z_s,s\leq t)$, the expression inside the expectation in the second term of  \eqref{3 terms}
  is bounded and continuous with
respect to the Skorokhod  topology
 (recall that Skorokhod continuity is preserved for $Z$ under the operation of  supremum over finite time horizons). Moreover, it vanishes when $ Z_s\equiv n$, $0\leq s\leq t$. 
This implies that 
    this term goes to $0$ as well. Finally, the expression whose absolute value we take in the  the third term  of \eqref{3 terms} is bounded by $1$, and vanishes unless $Z$ jumps at least once on $[0,s]$. This shows that the last term is bounded by $  \max_{n\leq N} \sup_{s\leq t}  \mathbf{P}^{T-s}_{0,n}  ( \sup_{w\leq t}\Delta Z_w>0)$, which goes to $0$  when $T\to \infty$. Note that for all three terms in \eqref{3 terms}, we are using the fact that, if $g(T)\geq 0$ is continuous in $T$ and $\lim_{T\to\infty}g(T) = 0$,  then, for each $\varepsilon> 0$, and $0< t\leq t_0$, by choosing $T$ sufficiently large, we have $\sup_{s\leq t}g(T-s)<\varepsilon$. That is to say, $\lim_{T\to\infty}\sup_{s\leq t}g(T-s) = 0$.

    Putting the pieces  together and choosing $N\in  \mathbb{N}_0$ large enough such that $B^N\leq \varepsilon$, we thus get
$$\limsup_{T\to \infty} \sup_{x\geq 0,s\leq t,n\in \mathbb{N}_0}|\texttt{P}^T_t[f_{\gamma,\theta,\varphi}](x,n,s)-\texttt{P}^\uparrow_t[f_{\gamma,\theta,\varphi}](x,n,s)|\leq 2\varepsilon.$$
Since $\varepsilon$ was arbitrary this shows the convergence of the semigroups \eqref{Fellercgce} in our setting which, together with the weak convergence of the initial configurations, gives the weak convergence of the associated processes on $[0,t_0]$. 
And since $t_0>0$ was chosen arbitrarily, this also completes the proof of Theorem \ref{skeletontospine}.
\end{proof}

\section*{Acknowledgements}

 Part of this work was carried out whilst AEK was visiting the Centre for Mathematical Modelling, Universidad de Chile and JF was visiting the Department of Mathematical Sciences at the University of Bath, each is grateful to the host institution of the other for their support.  
The authors would also like to thank the anonymous referees whose extensive reading of earlier versions of this paper led to many improvements.

\bibliography{biblio}{}
\bibliographystyle{plain}

\end{document}